\documentclass[11pt]{amsart}
\usepackage{verbatim,xcolor}
\usepackage{amscd, ifpdf,bm,amssymb,amsmath,amsthm}

\usepackage[mathscr]{eucal}
\ifpdf
  \usepackage{graphicx,float,wrapfig}
\fi

\newcommand\ie{\textit{i.e. }}

\newcommand\la{\lambda}
\newcommand\al{\alpha}

\newcommand\fg{\mathfrak g}

\newcommand\bC{\mathbb C}

\newcommand\bN{\mathbb N}
\newcommand\vh{\vspace{0.1in}}

\newcommand\disp{\displaystyle}

\newcommand\sform[2]{\langle #1,#2\rangle}

\newcommand\cf{\textit{cf }}

\newcommand{\bR}{\mathbb{R}}
\newcommand{\bZ}{\mathbb{Z}}

\newcommand\C[1]{\mathcal{#1}}
\newcommand\CL{\C L}
\newcommand\CO{\C O}
\newcommand\cCO{\C O^\vee}

\newcommand\CP{\mathcal P}
\newcommand\CS{\mathcal S}
\newcommand\CX{\C X}
\newcommand\CZ{\mathcal Z}
\newcommand\gL{\C L\!\!\!\C L}

\newcommand\bb[1]{\mathbb{#1}}
\newcommand\ovl[1]{\overline{#1}}
\newcommand\ovb[1]{\overline{#1}}
\newcommand\ovr[1]{\overline{#1}}

\newcommand\fk[1]{\mathfrak{#1}}

\newcommand\fm{\fk m}
\newcommand\fn{\fk n}
\newcommand\fp{\fk p}

\newcommand\ep{\epsilon}
\newcommand\om{\omega}

\newcommand\Om{\Omega}

\newcommand\wh[1]{\widehat{#1}}
\newcommand\wht[1]{\widehat{#1}}
\newcommand\wti[1]{\widetilde{#1}}
\newcommand\unb[2]{\underbrace{#2}_{#1}}

\newcommand\seq[2]{{#1_1,\dots,#1_#2}}
\newcommand\seqzero{{0,\dots,0}}
\newcommand\obr[2]{\overset{#1}{\overbrace{#2}}}

\newcommand\clrr{\color{red}}
\newcommand\clrblu{\color{blue}}

\newcommand\Triv{Triv}

\newcommand{\bc}{\begin{center}}
\newcommand{\ec}{\end{center}}

\usepackage{verbatim, hyperref, setspace}

\newcommand\tr{\operatorname{Tr}}
\newcommand\ad{\operatorname{ad}}
\newcommand\Ad{\operatorname{Ad}}
\newcommand\Ann{\operatorname{Ann}}
\newcommand\Ind{\operatorname{Ind}}
\newcommand\Hom{\operatorname{Hom}}

\newtheorem{theorem}{Theorem}[subsection]
\newtheorem*{theorem*}{Theorem}

\newtheorem{corollary}[theorem]{Corollary}
\newtheorem{conjecture}[theorem]{Conjecture}
\newtheorem{definition}[theorem]{Definition}
\newtheorem{example}[theorem]{Example}
\newtheorem{lemma}[theorem]{Lemma}
\newtheorem{proposition}[theorem]{Proposition}
\newtheorem{remark}[theorem]{Remark}

\newcommand\be[1]{\begin{#1}}
\newcommand\ee[1]{\end{#1}}
\begin{document}



\dedicatory{\flushright\normalsize To Roger Howe with admiration}
\title{Unipotent Representations and the Dual
        Pair Correspondence} 

\author{Dan Barbasch}
\address{
Department of Mathematics\\
Cornell University
}
\email{barbasch@math.cornell.edu}
\maketitle
  \section{Introduction}
This paper describes some properties of the unipotent representations
in relation to the $\Theta-$correspondence and rational functions on
coadjoint nilpotent orbits in the Lie algebra.

\bigskip
Let $\fk g_\bC$ be the complexification of a real reductive Lie algebra
$\fk g$, and $G$ a real reductive group with Lie algebra $\fk g$ and 
maximal compact subgroup $K\subset G.$ 
\begin{definition}
An irreducible $(\fk g,K)-$module $(\Pi,V)$ for a real reductive group
$G$ is called \textbf{unipotent} if
\begin{description}\label{d:unipotent}
\item[(1)] $\Ann\Pi\subset U(\fk g)$ is a maximal primitive ideal, 
\item[(2)] $(\Pi,V)$ is unitary.
\end{description}
\end{definition}  
Let $(G_1,G_2)$ be pair of groups which form a dual reductive pair,
and $\Pi_1$ a unipotent representation of $G_1.$ 
The question is when $\Pi_1$ occurs in the 
\textit{$\Theta-$correspondence},  
as introduced and studied in the work of Roger Howe.
{
The paper treats the case of complex groups viewed as real
groups; $\fk g$ is the Lie algebra of a complex group viewed as a real
group. A lot of the material is available for real groups, still in
progress.} The main reason for this restriction is that unipotent
representations are classified in the case of complex groups in the sense
that their Langlands parameters
are explicitly given in \cite{B1}, and the Theta correspondence
is also explicitly described in  \cite{AB1}.  We mainly treat the
cases of $Sp(2n,\bb C)\times O(m,\bb C)$;  $GL(n,\bb
C)\times Gl(n,\bb C)$ is straightforward. The nature of the  answer is that
for any unipotent representation $\Pi$, there is a sequence 
$$
G_0=G,\ G_1,\dots ,G_r,
$$ 
such that each $(G_i,G_{i+1})$ is a dual pair, 
and unipotent representations $\Pi_i$ so that $(\Pi_i,\Pi_{i+1})$ occur in the
$\Theta-$correspondence, and the last one $\Pi_r$ is 1-dimensional. The precise
conditions on $\Pi_i$ and the $G_i$ are given in Section
\ref{sec:metaplectic}, Theorem \ref{t:corresp}.  

\bigskip
The second theme is the relation to regular functions on coadjoint
orbits. To each unipotent representation one can associate a nilpotent
orbit $\C O\subset \fk g$ and a number $m(\Pi,\CO)$ called the
\textit{Asymptotic Support} and \textit{Multiplicity}
respectively. We use the (equivalent) versions of \textit{Associated
  Cycle} and \textit{Multiplicity} in \cite{V}. 
Let $Unip(\CO)$ be the set of unipotent representations
with asymptotic support $\CO.$ Let $e\in \CO$ be a representative,
$C_G(\CO):=C_G(e)$ be the centralizer, and $A(\CO):=C_G(e)/C_G(e)^0$
be the component group. One of the main results in \cite{V} is that
there is an (algebraic) representation $\psi(\Pi,\CO)$ of $C_G(\CO)$ 
such that the multiplicity of $\Pi$ is $\dim\psi,$ and 
$$
\Pi\mid_{K_\bb C}=R(\CO,\psi)-Y_\psi
$$
where $K$ is the  maximal compact subgroup
$K\subset G,$ so $K_\bb C$ is equivalent to $G,$ 
$R(\CO,\psi)$ the space of regular sections on $\CO$ transforming
according to $\psi$ viewed as  $G-$module, and $Y_\psi$ is a
$K_\bb C-$representation with support in nilpotent orbits in the closure of
$\CO,$ strictly smaller than $\CO$.  
As already mentioned in \cite{V}, it is
conjectured that there is a $1-1$ correspondence
$\psi\longleftrightarrow \Pi_\psi$ between $\wh{A(\CO)}$ and 
$Unip(\CO)$ such that
$$
\Pi_\psi\mid_K\cong R(\CO,\psi),
$$
in particular $Y_\psi=0.$ 
We establish this conjecture for a large class of nilpotent
orbits in the classical Lie groups. The relation follows for more
general orbits from certain geometric properties of the
resolution of nilpotent orbits for classical groups in \cite{KP1}. We
will pursue this in a later paper.

\bigskip
The correspondence between orbits and unipotent representations is
conjectured to hold for general groups. The last sections 
investigate its validity for the simply connected groups $Spin(n,\bb
C),$ and the case of $F_4.$ The groups of type $E$ will be considered
in a different paper.  

\bigskip
{Different properties of unipotent representations are considered in 
\cite{Moe} and \cite{MR}. There is very little overlap with the
results in this paper.}

\bigskip
{
One of the aims  of the paper is to highlight the impact that
Roger Howe's work had on my own work. 
I first met Roger Howe at a conference  in Luminy in 1978. At the time
I knew the work of Rallis and Schiffman and Kashiwara-Vergne on the
dual pairs correspondence when one of the groups was compact. The case
when neither group was compact seemed completely unreachable. I was
stunned by the results that Roger presented for this latter case. 
Some ten years later, I understood enough to write a paper joint with
J. Adams, \cite{AB1}, where we described the
correspondence for complex groups in detail. Extensions of these
results to some real classical groups appear in \cite{AB2}. 
}
{
One of my students, Shu-Yen Pan,
investigated the correspondence in the case of p-adic
groups, and another student, Daniel Wong, investigated an extension of the
Theta correpondence.

Along different lines, at the same time that I started my
collaboration with Adams, I met and started to collaborate
with Allen Moy.  Another ten years later we
gave a new proof of the Howe conjecture for p-adic groups. 

\subsection*{Acknowledgments} The author was supported by an NSA grant.

}

\section{Unipotent Representations}

In this section we review the basics of the representation theory of
admisisble representations of complex groups viewed as real groups.

\subsection{Complex Groups Viewed as Real Groups}
This material is taken from \cite{V1}. Modules are all admissible
$(\fk g_\bC,K)-$modules. 
\begin{lemma}
 Let $\fk g$ be a complex Lie algebra, and let $\fk g_0$ be the same
 algebra  viewed as a real Lie  algebra. Then the complexification
 $\fk g_\bC$ canonically identifies with
$$
\fk g_\bC=\fk g_L+ \fk g_R.
$$  
The summand $\fk g_L$ is isomorphic to $\fk g,$ and $\fk g_R$  to
the complex conjugate algebra. The $\star$-antiautomorphism on $\fk
g_\bC$ interchanges the two summands.
\end{lemma}
\begin{proof}
Let $j$ be the multiplication by $\sqrt{-1}$ on $\fk g.$ This is a
real linear transformation on $\fk g_0$ and so defines a complex
linear transformation $J$ on $\fk g_\bC= \fk g_0 +i\fk g_0,$
satisfying $J(X+iY)=JX+iJY$ for $X,Y\in\fk g_0.$ Then
\begin{equation*}
  \begin{aligned}
    &\fk g_L=\{\frac12(X-iJX)\mid X\in\fk g_\bC\},\\
    &\fk g_R=\{\frac12(X+iJX)\mid X\in\fk g_\bC\},\\
  \end{aligned}
\end{equation*}
are complex subalgebras of $\fk g_\bC.$   
The algebra $\fk g_L$ is isomorphic to $\fk g$ via 
$$
(\al+i\beta)-iJ(\al +i\beta)\mapsto \al +j\beta,\qquad \al,\beta\in\fk g_0.
$$
The algebra $\fk g_R$ is isomorphic to $\fk g$ with conjugate linear
multiplication.
\end{proof}

\subsection{Langlands Parameters} 
We use the standard realizations of the classical
groups, roots, positive roots and simple roots. Let 
\begin{itemize}
\item $\theta$ Cartan involution, $K$ the fixed points of $\theta$,
  $\fk g=\fk k +\fk p$ the Cartan decomposition, 
\item  $\fk b=\fk h +\fk n$ a Borel subalgebra, 
\item $\fk h=\fk t +\fk a$ a Cartan subalgebra, $\fk t\subset\fk k$, 
$\theta\mid_\fk a=-Id,$
\item $W$ the Weyl group of $(\fk g,\fk h),$
\item $X(\mu,\nu)=\Ind_B^G(\bC_\mu\otimes\bC_\nu)$  standard module, 
\item $L(\mu,\nu)$, the unique subquotient containing
  $V_\mu\in\widehat K,$
\item $\la_L=(\mu+\nu)/2$ and $\la_R=(-\mu+\nu)/2.$ 
\end{itemize}
 
{The parameters of unipotent representations have real $\nu$, so we
  will assume this in the rest of the paper.}
\begin{theorem}\ 

  \begin{enumerate}
  \item $L(\la_L,\la_R)\cong L(\la_L',\la_R')$ if and only if there is
    a $w\in W$ such that $w\cdot (\la_L,\la_R)=(\la'_L,\la'_R).$
\item $L(\la_L,\la_R)$ is hermitian if and only if there is $w\in W$
  such that \newline $w\cdot(\mu,\nu)=(\mu,-\nu).$  
  \end{enumerate}
\end{theorem}
We will write the parameter in column form as 
$
\begin{pmatrix}
  \la_L\\\la_R
\end{pmatrix}
$ mainly for display reasons.

\subsection{Parameters of Unipotent Representations}
\label{sec:upar}
{
We rely on \cite{BV2} and \cite{B1}. For each $\CO\subset\fk g$ we
will give an infinitesimal character $(\la_\CO,\la_\CO)$, and a set of
$(\la_\CO,w\la_\CO)$ such that  $\{L(\la_{\C O},w\la_{\C O})\}$ are
the unipotent representations with asymptotic support $\CO.$  
In all cases $\la_{\C O}$ and $-\la_{\C O}$ are in the same $W-$orbit.

\bigskip
\noindent\textbf{Main Properties of }$\boldsymbol{\la_{\C O}}$. 
Suppose $\Pi$ is an
irreducible representation with infinitesimal character
$(\la_\CO,\la_\CO).$ Then $\la_\CO$ and $\Pi$ must satisfy:

\begin{enumerate}
\item $\Ann\Pi\subset U(\fk g)$ is the maximal primitive ideal
  $\C I_{\la_\CO}$ with infinitesimal character $(\la_{\C O},\la_\CO)$,

\item {$\mid \{\Pi\ :\ \Ann\Pi=\C I_{\la_\CO}\}\mid =\mid \widehat{A(\C
    O)}\mid,$ where $A(\CO)$ is the component group of the centralizer
  of an $e\in \C O,$}  

\item $\Pi$ unitary.

\end{enumerate}

\be{remark} The component group $A(\C O)$ depends on the isogeny class of
  $G,$ which will be a classical group $Sp(2n,\bC)$, $SO(m,\bC)$ or
  $O(m,\bC).$ 
\ee{remark}

}
The notation is as in \cite{B1}. The choices of $\la_{\C O}$ satisfying
(3) rely on the determination of the unitary dual for classical groups
in \cite{B1}. For special orbits $\C O$ whose dual $\cCO$ is even,
$\la_{\C O}$ is half the semisimple element of the Lie triple
corresponding to the dual orbit, $\la_\CO=h(\cCO)/2$. 
For the other orbits we need a case-by-case analysis. The
parameter will always have integer and half-integer coordinates, and the
corresponding system of integral coroots is maximal. 

\begin{definition}
A special orbit $\CO$ (in the sense of Lusztig) 
is called \textbf{stably trivial} if Lusztig's
quotient $\ovl{A}(\CO)$ equals the full component group $A(\CO).$
\end{definition}
For a definition and discussion of $\ovl{A}(\CO),$ see \cite{L},
chapter 13.

\medskip
{The partitions in the next examples denote rows}.
\begin{example}
$\CO=(2222)\subset sp(8)$ is stably
  trivial, $A(\CO)=\ovl{A(\CO)}\cong\bZ_2$, $\la_\CO=(2,1,1,0).$ 
In this case $\cCO$ corresponds to the partition $(531),$ and
$\la_\CO=h(\cCO)/2.$ 

$\CO=(222)\subset sp(6)$ has dual orbit $\cCO$ corresponding to $(331)$
but is not stably trivial; $A(\CO)\cong\bZ_2,$ while $\ovl{A(\C O)}\cong 1$.
In this case $h(\cCO)/2=(1,1,0),$ and for
this infinitesimal character, conditions (1) and (3) are satisfied,
but (2) is not satisfied. The choice of infinitesimal character in
this case will be $\la_\CO=(3/2,1/2,1/2).$ There are two parameters, 
$$
\be{pmatrix}
\la_L\\ \la_R
\ee{pmatrix}=
\be{pmatrix}
3/2&1/2&1/2\\ 
3/2&1/2&1/2
\ee{pmatrix}\qquad\text{ and }\qquad
\be{pmatrix}
3/2&1/2&1/2\\ 
3/2&1/2&-1/2
\ee{pmatrix}
$$  
\end{example}
\subsection{Type A}\label{sec:a} The group $G$ is $GL(n).$
Nilpotent orbits are determined by their
Jordan canonical form. An orbit is given by a partition, \ie a 
sequence of numbers in decreasing order $\CO\longleftrightarrow (n_1,\dots
,n_k)$ that add up to $n.$  Let $(m_1,\dots ,m_l)$ be the dual
partition. The component group of $\C O$ is trivial. 
The infinitesimal character is 
$$
\la_\CO=\bigg(\frac{m_1-1}{2},\dots ,-\frac{m_1-1}{2},\dots
,\frac{m_l-1}{2},\dots, -\frac{m_l-1}{2}\bigg).
$$
The orbit is induced from the trivial orbit on the Levi component $\fm$
of a parabolic subalgebra  $\fp=\fm +\fn$ with 
$\fm=gl(m_1)\times \dots \times gl(m_l).$ The corresponding
unipotent representation is spherical and induced irreducible from the
trivial representation on the same Levi component. {All orbits are
{\it special } and {\it stably trivial}}.
\subsection{Type B}\label{sec:b} 
{
We describe the case $SO(2m+1)$. For $O(2m+1)$ there are twice the
parameters, the parameters for $SO$  are tensored with $sign$.

\medskip 
A nilpotent orbit is determined by its
Jordan canonical form (in the standard representation). Then $\CO$  is
parametrized by a partition $\CO\longleftrightarrow (n_1,\dots ,n_k)$
of $2m+1$ such that every even entry occurs an even number of times. 
Let $(m'_0,\dots ,m'_{2p'})$ be the transpose partition (add an $m'_{2p'}=0$ if
necessary, in order to have an odd number of terms).  If $\CO$ is
represented by a tableau, these are the sizes of the columns in
decreasing order.
If there are any $m'_{2j}=m'_{2j+1}$, then pair them
together and remove them from the partition. 
Then relabel and pair up the remaining columns $(m_0)(m_1,m_2)\dots
(m_{2p-1}m_{2p}).$ The members of each 
pair have the same parity and $m_0$ is odd. $\la_\CO$ is given by the
coordinates 
\begin{equation}
  \label{eq:unipb}
\aligned
(m_0)&\longleftrightarrow (\frac{m_0-2}{2},\dots ,\frac12),\\
(m'_{2j}=m'_{2j+1})&\longleftrightarrow (\frac{m_{2j}'-1}{2},\dots , 
-\frac{m_{2j}'-1}{2})\\
(m_{2i-1}m_{2i})&\longleftrightarrow (\frac{m_{2i-1}}{2},\dots , -\frac{m_{2i}-2}{2}).
\endaligned  
\end{equation}
}

In case $m'_{2j}=m'_{2j+1},$ $\CO$ is induced from an orbit
$$
\CO_\fk m\subset\fm=so(*)\times gl\big(\frac{m'_{2j}+m'_{2j+1}}{2}\big)
$$ 
where $\fk m$ is the
Levi component of a parabolic subalgebra $\fp=\fm +\fn$. $\CO_\fk m$ is 
the trivial nilpotent on the $gl-$factor. The
component groups satisfy $A_G(\CO)\cong A_M(\CO_{\fk m}).$ 
Each unipotent representation is
unitarily induced from a unipotent representation attached to 
$\CO_\fk m.$  

Similarly if some $m_{2i-1}=m_{2i},$ then $\CO$ is
induced from a  
$$
\CO_{\fk m}\subset \fk m\cong so(*)\times
gl(\frac{m_{2i-1}+m_{2i}}{2})\quad (0)\quad  \text{ on the
  gl-factor.}
$$ 
$A_G(\CO)\not\cong A_M(\CO_{\fk m}),$ but
each unipotent representation is (not necessarily unitarily) induced
irreducible from a representation on the Levi component
$\fk m$, unipotent on
$so(*)$, and a character on the gl-factor. 

\medskip

{The {\it stably trivial} orbits are the
ones such that every odd sized part appears an even number of
times, except for the largest size}. An orbit is called triangular if it has
partition 
$$
\CO\longleftrightarrow(2m+1,2m-1,2m-1,\dots,3,3,1,1).
$$ 


We give the explicit Langlands parameters of the unipotent 
representations. There are $\mid A_G(\CO)|$ distinct representations. 
Let 
$$
(\underset{r_k}{\underbrace{k,\dots ,k}},\dots
,\underset{r_1}{\underbrace{1,\dots 1}})
$$ 
be the rows of the Jordan form of the
nilpotent orbit. The numbers $r_{2i}$ are even. 
The reductive part of the centralizer (when $G$ is the orthogonal
group) of the nilpotent element is a product of $O(r_{2i+1})$, and
$Sp(r_{2j})$.

The columns are paired as in (\ref{eq:unipb}). The pairs
$(m'_{2j}=m'_{2j+1})$ contribute to the spherical part of the
parameter,
\begin{equation}
  \label{eq:unipbprime}
(m'_{2j}=m'_{2j+1})\longleftrightarrow 
\begin{pmatrix}
  \la_L\\ \la_R
\end{pmatrix}
=
\begin{pmatrix}\frac{m'_{2j}-1}{2}&,&\dots&,& 
&-\frac{m'_{2j}-1}{2}\\
\frac{m'_{2j}-1}{2}&,&\dots&,& &-\frac{m'_{2j}-1}{2}
\end{pmatrix}  .
\end{equation}
The singleton $(m_0)$ contributes to the spherical part,
\begin{equation}
\label{eq:unipbzero}
(m_0)\longleftrightarrow
 \begin{pmatrix}
\frac{m_{0}-2}{2}&,&\dots&,&\frac{1}{2}\\
\frac{m_{0}-2}{2}&,&\dots&,&\frac{1}{2}
\end{pmatrix}  . 
\end{equation}
Let $(\eta_1,\dots ,\eta_{p})$ with $\eta_i=\pm 1,$ one for each
$(m_{2i-1},m_{2i})$. An $\eta_i=1$ contributes to the spherical part
of the parameter, with coordinates as in (\ref{eq:unipbprime}) and (\ref{eq:unipbzero}). An
$\eta_i=-1$ contributes
\begin{equation}
\label{eq:unipbep}
\begin{pmatrix}
\frac{m_{2i-1}}{2}&,&\dots&,&\frac{m_{2i}+2}{2}&\frac{m_{2i}}{2}&,&\dots
&,&-\frac{m_{2i}-2}{2}\\
\frac{m_{2i-1}}{2}&,&\dots&,&\frac{m_{2i}+2}{2}&\frac{m_{2i}-2}{2}&,&\dots
&,&-\frac{m_{2i}}{2}
\end{pmatrix}.
\end{equation}
{If $m_{2p}=0,$ $\eta_p=1$ only for $SO$.}
\subsection{Explanation}
  \begin{enumerate}
  \item Odd sized rows contribute a $\bZ_2$ to $A(\CO),$ 
    even sized rows a $1.$
\item When there are no $m'_{2j}=m'_{2j+1},$ every row size occurs. 
{The inequalities
$$\dots (m_{2i-1}\ge m_{2i})>(m_{2i+1}\ge  m_{2i+2})\dots
$$} 
imply that there are $m_{2i}-m_{2i+1}$ rows of size $2i+1.$ Each
pair $(m_{2i-1}\ge m_{2i})$ contributes exactly 2 parameters
corresponding to the $\bZ_2$ in $A(\CO)$. 
\item The pairs $(m'_{2j}=m'_{2j+1})$ lengthen the sizes of the rows
  without changing their parity. The component group does not change,
  they do not affect the  number of parameters.
  \end{enumerate}

\bigskip
As already mentioned, when $G=O(2m+1,\bC)$ the unipotent
representations are obtained from those of $SO(2m,\bC)$ by lifting
them to $O(2m,\bC)$, and also tensoring with $sgn$.

\subsection{Type C}\label{sec:c} 

{
A nilpotent orbit is determined by its
Jordan canonical form (in the standard representation). It is
parametrized by a partition \newline $\CO\longleftrightarrow(n_1,\dots ,n_k)$ 
of $2n$ such that
every odd part occurs an even number of times. 
Let $(c'_0,\dots ,c'_{2p'})$ be the dual partition (add a $c'_{2p'}= 0$ if
necessary in order to have an odd number of terms). As in type B,
these are the sizes of the columns of the tabelau corresponding to $\CO$.
If there are any $c'_{2j-1}=c'_{2j}$ pair them up and
remove them from the partition. Then relabel and pair
up the remaining columns $(c_0c_1)\dots
(c_{2p-2}c_{2p-1})(c_{2p}).$ The members of each pair have the same
parity. The last one, $c_{2p},$ is always even. Then
form a parameter
\begin{align}
(c'_{2j-1}=c'_{2j})&\longleftrightarrow (\frac{c_{2j}-1}{2},\dots , 
-\frac{c_{2j}-1}{2}),\label{eq:c1}\\
(c_{2i}c_{2i+1})&\longleftrightarrow (\frac{c_{2i}}{2},\dots , 
-\frac{c_{2i+1}-2}{2}),\label{eq:c2}\\
c_{2p}&\longleftrightarrow (\frac{c_{2p}}{2},\dots , 1).\label{eq:c3}
\end{align}
}
{
The nilpotent orbits and the unipotent representations have the same
properties with respect to these pairs as the corresponding ones in
type B. 

{The {\it stably trivial} orbits are the
ones such that every even sized part appears an even number of
times}. 

An orbit is called triangular if it corresponds to the
partition \newline
$(2m,2m,\dots,4,4,2,2).$ 


\bigskip
We give a parametrization of the unipotent representations in terms of
their Langlands parameters. There are $\mid A_G(\CO)\mid$
representations. 

Let 
$$
(\underset{r_k}{\underbrace{k,\dots ,k}},\dots
,\underset{r_1}{\underbrace{1,\dots ,1}})
$$ 
be the rows of the Jordan form of the
nilpotent orbit. The numbers $r_{2i+1}$ are even. 
}
{
The reductive part of the centralizer of the
nilpotent element is a product of $Sp(r_{2i+1})$, and $O(r_{2j})$. 

The elements $(c'_{2j-1}=c'_{2j})$ and $c_{2p}$ contribute to the
spherical part of the parameter as in (\ref{eq:unipbprime}) and
(\ref{eq:unipbzero}). 
Let $(\eta_1,\dots ,\eta_p)$ be such that $\eta_i=\pm 1,$ one for each
$(c_{2i},c_{2i+1}).$  An $\eta_i=1$ contributes to the spherical
part, according to the infinitesimal character. An $\eta_i=-1$ contributes
\begin{equation}
\label{eq:unipcep}
\begin{pmatrix}
&\frac{c_{2i}}{2}&,&\dots&,&\frac{c_{2i+1}+2}{2}&\frac{c_{2i+1}}{2}&\dots
&,&-\frac{c_{2i+1}-2}{2}\\ 
&\frac{c_{2i}}{2}&,&\dots&,&\frac{c_{2i+1}+2}{2}&\frac{c_{2i+1}-2}{2}&\dots &,&-\frac{c_{2i+1}}{2}\\ 
\end{pmatrix}.
\end{equation}
}

\bigskip
The explanation is similar to type B.
\subsection{Type D}\label{sec:d} 

{
We treat the case $G=SO(2m).$ 
A nilpotent orbit is determined by its
Jordan canonical form (in the standard representation). It is
parametrized by a partition $\CO\longleftrightarrow (n_1,\dots ,n_k)$ 
of $2m$ such that
every even part occurs an even number of times. 
Let $(m'_0,\dots ,m'_{2p'-1})$ be the dual partition (add a $m'_{2p'-1}=0$ if
necessary), the sizes of the columns of the tableau corresponding to
$\CO.$  If there are any $m'_{2j}=m'_{2j+1}$ pair them up and
remove from the partition. 
Then pair up the remaining columns $(m_0,m_{2p-1})(m_1,m_2)\dots
(m_{2p-3},m_{2p-2}).$ The members of each pair have the same parity and
$m_0, m_{2p-1}$ are both even. The infinitesimal character is
\begin{equation}
  \label{eq:unipd}
\aligned
(m'_{2j}=m'_{2j+1})&\longleftrightarrow (\frac{m'_{2j}-1}{2}\dots , -\frac{m'_{2j}-1}{2})\\
(m_0m_{2p-1})&\longleftrightarrow (\frac{m_0-2}{2},\dots ,-\frac{m_{2p-1}}{2}),\\
(m_{2i-1}m_{2i})&\longleftrightarrow (\frac{m_{2i-1}}{2}\dots , -\frac{m_{2i}-2}{2}) 
\endaligned  
\end{equation}
}
{
The nilpotent orbits and the unipotent representations have the same
properties with respect to these pairs as the corresponding ones in
type B. 
An exception occurs for $G=SO(2m)$ when the partition is formed of pairs
$(m'_{2j}=m'_{2j+1})$ only. In this case there are two nilpotent
orbits corresponding to the partition. There are also two nonconjugate
Levi components of the form $gl(m'_0)\times gl(m'_2)\times \dots
gl(m'_{2p'-2})$ of parabolic subalgebras. There are two unipotent
representations each induced irreducible from the trivial
representation on the corresponding Levi component. 

{The {\it stably
trivial} orbits are the ones such that every even sized part appears
an even number of times}. 

A nilpotent orbit is triangular if it
corresponds to the partition\newline $(2m-1,2m-1,\dots ,3,3,1,1).$ 

}
{
The parametrization of the unipotent representations follows types
B,C, with the pairs $(m'_{2j}=m'_{2j+1})$ and $(m_0,m_{2p-1})$
contributing to the spherical part of the parameter only. Similarly for
$(m_{2i-1},m_{2i})$ with $\ep_i=1$ spherical only, while $\ep_i=-1$
contributes analogous to (\ref{eq:unipbep}) and (\ref{eq:unipcep}).

The explanation parallels that for types B,C. 

\bigskip
When $G=O(2m,\bC)$ the unipotent representations are obtained from
those of $SO(2m,\bC)$ by lifting them to $O(2m,\bC)$, and also
tensoring with $sgn$. In the case when all $m'_{2j}=m'_{2j+1}$ the 
representations associated to the two nilpotent orbits have the same
lift, and it is invariant under tensoring with $sgn$. Otherwise
tensoring with $sgn$ gives inequivalent unipotent representations.
}
\section{Theta Correspondence}\label{sec:metaplectic}

We deal with the complex pairs
$G_1\times G_2$ where one group is orthogonal the other
symplectic. The results are from \cite{AB1}. Let $V_i$  for $i=1,2$ be
spaces endowed with nondegenerate forms, one symplectic the other
orthogonal. Then $\C W=V_1\otimes V_2,$ is symplectic, and $G_1\times
G_2:=G(V_1)\times G(V_2)$ is a dual pair. Up to isomorphism,
$(G_1,G_2)$ is $(O(n,\bC),Sp(2m,\bC))$ or $(Sp(2m,\bC),O(n,\bC)).$ 
Let $\tau=0,1$ depending whether $n$ (for the orthogonal group) is even or odd. 

\subsection{Complex Pairs}\label{sec:cx} 
Let $\big(V_0,\langle\ ,\ \rangle_0\big)$ be a real symplectic vector
space. We  can view $\langle \ ,\ \rangle_0$ as a linear map $\C
J_0:V_0\longrightarrow V_0'$ ($V_0'$ the linear dual of $V_0$) 
satisfying $\C J_0^t=-\C J_0,$ so that the
symplectic form is given by

\begin{equation}
  \label{eq:sform}
  \langle v_1,v_2\rangle_0=(\C J_0v_2)(v_1).
\end{equation}

Let $V_\bC=V_0+iV_0$ be the complexification of $V_0$, and $\sform{\ }{\ }$ be
the complexification of $\sform{\ }{\ }_0.$ It satisfies
\begin{equation}
  \label{eq:sform1}
  \sform{v_1+iv_2}{w_1+iw_2}=\big({\sform{v_1}{w_1}}_0 -
{\sform{v_2}{w_2}}_0\big)
+ i \big({\sform{v_1}{w_2}}_0+\sform{v_2}{w_1}_0\big)
\end{equation}
The complex symplectic Lie algebra $\fk g_0:=sp(V_\bC)$ is the algebra
preserving $\sform{\ }{\ }.$ Let $\C V=V_0\oplus V_0$ be the real
vector space identified with $V_\bC$ in the usual way,
$v_1+iv_2\longleftrightarrow (v_1,v_2)$. An element $a=\al
+i\beta\subset sp(V_\bC)$ is then
\begin{equation}
  \label{eq:sform1.1}
\al +i\beta\longleftrightarrow     
  \begin{bmatrix}
    \al&-\beta\\\beta&\al
  \end{bmatrix}.
\end{equation}
The real part and imaginary part of $\sform{\ }{\ }$ are symplectic
(nondegenerate) forms on $\C V;$ denote them by $\sform{\ }{\ }_{re}$
and $\sform{\ }{\ }_{im}.$ In terms of skew maps from $\C V$ to $\C
V',$ they are
\begin{equation}
  \label{eq:sform3}
  \begin{aligned}
    &\sform{\ }{\ }_{re}\longleftrightarrow
    \begin{bmatrix}
      \C J_0&0\\0&-\C J_0
    \end{bmatrix},\\
    &\sform{\ }{\ }_{im}\longleftrightarrow
    \begin{bmatrix}
      0&\C J_0\\\C J_0&0
    \end{bmatrix}.
  \end{aligned}
\end{equation}
View $sp(V_\bC)$ as a real Lie algebra. Then $sp(V_\bC)$ embeds in $\big(sp(\C
V),\sform{\ }{\ }_{re,im}\big)$ via formula (\ref{eq:sform1.1}). We
choose $\sform{\ }{\ }_{re},$ and note that $sp(V_\bC),\ sp(\C V)$ are
invariant under transpose, and the  inclusion $sp(V)\subset sp(\C V)$
commutes with the transpose map. We will view $sp(V)$ as the Lie subalgebra of
$sp(\C V)$ under the inclusion (\ref{eq:sform1.1}).

The Cartan decomposition  of (the real Lie algebra) $\fk g_0:=sp(V_\bC)$ is 
\begin{equation}
  \label{eq:cartan}
  \begin{aligned}
&\fk g_0=\fk k_0 +\fk s_0,  \\
&\fk k_0=\{ \al +i\beta\ :\ (\al +i\beta)+(\al-i\beta)^t=0\},\\
&\fk s_0=\{ \al +i\beta\ :\ (\al +i\beta)-(\al-i\beta)^t=0\}.
\end{aligned}
\end{equation}
Similarly the Cartan decomposition of $\fk g_\C V:=sp(\C V)$ is 
\begin{equation}
  \label{eq:bigcartan}
  \begin{aligned}
&\fk g_\C V=\fk k_\C V +\fk s_\C V,  \\
&\fk k_\C V=\{A\in sp(\C V)\ :\ A + A^t=0\},\\
&\fk s_\C V=\{A\in sp(\C V)\ :\ A-A^t=0\}.
\end{aligned}  
\end{equation}
In particular, $\fk k_0\subset \fk k_{\C V}$ and $\fk s_0\subset\fk
s_\C V.$ 
\subsubsection{A Variant} Let $(V,\langle\ ,\ \rangle)$ be a
symplectic complex space with form corresponding to 
$\disp{\C J=
\begin{bmatrix}
  0&I\\-I&0
\end{bmatrix}.}$ Then $sp(2n,\bC)$ embeds in $sp(4n,\bC)$ with the
usual symplectic form as
$$
\begin{bmatrix} 
  \al&\beta\C J\\-\C J\beta&-\al^t
\end{bmatrix}
$$
where $\al,\beta\in sp(2n)_c$ is the compact real form of
$sp(2n,\bC).$ Multiplication by $\sqrt{-1}$ corresponds to 
$$
m_{\sqrt{-1}}:
\begin{bmatrix}
  \al&0\\0&-\al^t
\end{bmatrix}
\longrightarrow 
\begin{bmatrix}
  0&\al\C J\\-\C J\al&0
\end{bmatrix}.
$$
The complexification $sp(2n.\bC)_c\subset sp(4n,\bC)$ is the same,
replace \newline $\al,\beta\in sp(2n,\bR)$ by $\al,\beta\in sp(2n,\bC).$

\subsection{Oscillator Representation} Let $\Om_{\C V}=\Om_++\Om_-$  
be the oscillator representation of $sp(\C V)=sp(4n,\bb R).$ 

The following is well known (and straightforward).

\begin{theorem}\ 

  \begin{enumerate}
  \item The pairs $(O(m,\bb C),Sp(2n,\bb C))\subset Sp(2mn,\bb
    C)\subset Sp(4mn,\bb R)$ are dual pairs.
\item The restrictions to $sp(2n,\bb C)$ of $\Om_\pm$ are irreducible
  and equal to the two representations of $sp(2n,\bb C)$ corresponding
  to the minimal nontrivial nilpotent orbit. 
  \end{enumerate}
\end{theorem}
\subsection{Infinitesimal Character} 

\begin{proposition}\label{p:infl}

Suppose $\pi_1$ corresponds to $\pi_2$ in the dual pair
correspondence for $(G_1,G_2)$. Let  $( \lambda_1,\lambda_1')$ 
be the infinitesimal character of
$\pi_1$.  Write the infinitesimal character of $\pi_2$ as
$( \lambda_2,\lambda_2')$.
Then for $\lambda_2$ and $\lambda_2'$ we may take
$\lambda_2=\lambda_1\cdot\tilde\lambda$ and
$\lambda_2'=\lambda_1'\cdot\tilde\lambda$
with $\tilde\lambda$ as follows:
\begin{enumerate}
\item $(O(m),Sp(2n)),\ [\tfrac m2]\le n$:
$\tilde\lambda=(n-m/2,n-m/2-1,\dots,1-\tau/2),$
\item $(Sp(2m),O(n)),\ m\le[\tfrac n2]$:
$\tilde\lambda=(n/2-m-1,n/2-m-2,\dots,\tau/2),$
\item $(GL(m),GL(n)),\ m\le n$:
$\tilde\lambda=
\frac12(n-m-1,n-m-3,\dots,-n+m+1)$.
\end{enumerate}
Here  $\cdot$ indicates concatenation of sequences.
\end{proposition}
\subsection{Langlands Parameters} 

$K-$types for $O(n)$ are parametrized as in Weyl's work using the standard
embedding  $O(n)\subset U(n).$ 
An irreducible representation of $O(n)$ is parametrized by 
$$
(a_1,\dots, a_{k},0\dots 0;\ep)
$$
with $\dots a_i\ge a_{i+1}\ge \dots a_k>0$ integers and $\ep =0,1$ so that
the representation is the $O(n)-$irreducible component generated by
the highest weight of the representation of $U(n)$ with highest weight
$$
(a_1,\dots ,a_k,{\underbrace{1,\dots ,1}_{n-2k},0,\dots 0}).
$$
The basic cases for the correspondece are summarized in the next
proposition. The general case is the next theorem. 
\begin{proposition}\label{p:basic}[Proposition 2.1,\cite{AB1}, 
Basic Cases (Type I)]
\begin{enumerate}
\item Let $triv$ be the trivial representation of $O(m,\bb C)$. Then for 
any $n\ge0$, $\Theta(triv)$ is the unique irreducible spherical 
representation of $Sp(2n,\bb C)$ with infinitesimal character given by 
Proposition \ref{p:infl}. Thus\newline 
$\Theta(triv)=L(0,\nu)$ with $\nu=(m-2,m-4,\dots,m-2n)$. In terms of
$\la_L,\la_R,$ the parameter is
$$
\be{pmatrix}
\la_L\\ \la_R
\ee{pmatrix}=
\begin{pmatrix}
m/2-1,\dots ,m/2-n\\
m/2-1,\dots  ,m/2-n 
\end{pmatrix}
$$

\item Let $triv$ be the trivial representation of $Sp(2m,\bb C)$. 

(a) For any even $n\ge 0$, $\Theta(triv)$ is the unique irreducible spherical 
representation of $O(n,\bb C)$ with infinitesimal character given by 
Proposition \ref{p:infl}. 
Thus 
$\Theta(triv)=L(0,\nu)$ with 
$$
\nu=(2m,2m-2,\dots,2m-n+2).
$$ 
In terms of $\la_L,\la_R$ the parameter is
$$
\be{pmatrix}
\la_L\\ \la_R
\ee{pmatrix}=
\be{pmatrix}
\la_L\\ \la_R
\ee{pmatrix}=
\begin{pmatrix}
m,m-1,\dots m-n/2+1\\
 m,m-1,\dots m-n/2+1 
\end{pmatrix}
$$

(b) If $n$ is odd, then $triv$ occurs in the correspondence with
$O(n,\bb C)$ if and only if 
$n>2m$, and the same conclusion as in (a) holds with 
$$
\nu=(2m,2m-2,\dots,2,-1,-3,\dots,2m-n+2).
$$ 
In terms of $\la_L,\la_R,$
the parameter is
$$
\be{pmatrix}
\la_L\\ \la_R
\ee{pmatrix}=
\begin{pmatrix}
m,m-1,\dots ,1,-1/2,\dots ,m-n/2+1\\
m,m-1,\dots ,1,-1/2,\dots ,m-n/2+1 
\end{pmatrix}
$$

\item The $sgn$ representation of $O(m,\bb C)$ occurs in the 
representation correspondence with $Sp(2n,\bb C)$ if and only if $n\ge 
m$.

(a) For $n=m,$ $\Theta(sgn)$ is the unique irreducible representation 
of $Sp(2m,\bb C)$ with lowest K--type equal to the K--type pairing with 
the $sgn$ representation of $O(m)$  (cf. Proposition 1.4 in \cite{AB1}), and 
infinitesimal character given by Proposition \ref{p:infl}. 
Thus $\Theta(sgn)=L(\mu,\nu)$ with
$$
\begin{pmatrix}
\mu\\
\nu
\end{pmatrix}
=
\begin{pmatrix}
1&1&\dots &1\\
m-1&m-3&\dots&-m+1    
\end{pmatrix}
$$
In terms of $\la_L,\la_R,$ the parameter is
$$
\be{pmatrix}
\la_L\\ \la_R
\ee{pmatrix}=
\begin{pmatrix}
m/2+1/2,\dots ,-m/2+1/2\\
m/2-1/2,\dots ,-m/2-1/2  
\end{pmatrix}
$$
(b) For $n>m$, let $P$ be a parabolic subgroup of $Sp(2n,\bb C)$ with Levi
factor $M=GL(n-m,\bb C)\times Sp(2m,\bb C)$. Then $\Theta(sgn)$ is the unique
irreducible subquotient of
$$
Ind_{P}^{Sp(2n,\bb C)} (|det|^{n+1}\otimes\Theta_m(sgn))
$$
containing the lowest K--type  of this induced  representation. 
Here $\Theta_m$ denotes the $\Theta$--lift from $O(m,\bb C)$ to 
$Sp(2m,\bb C)$. Explicitly: 
$\Theta(sgn)=L(\mu,\nu)$ with 
$$
\begin{aligned}
\mu&=(\overbrace{1,\dots ,1}^m,0,\dots ,0),\\
\nu&=(\overbrace{m-1,m-3,\dots ,-m+1}^m,2n-m,2n-m-2,\dots,m+2).
\end{aligned}
$$
In terms of $\la_L,\la_R,$ the parameter is
$$
\be{pmatrix}
\la_L\\ \la_R
\ee{pmatrix}=
\begin{pmatrix}
m/2-1/2,\dots ,-m/2+3/2,n-m,\dots ,m/2+1\\
m/2-3/2,\dots ,-m/2+1/2,n-m,\dots ,m/2+1  
\end{pmatrix}
$$
\end{enumerate}
\end{proposition}

\begin{theorem}[\cite{AB1} Theorem 2.8: Explicit dual pair correspondence (Type
  I)]
\label{t:abmain}
Fix $\tau=0,1$ and consider a family of dual pairs 
$(G_1(m),G_2(n))=(O(2m+\tau,\bb C),Sp(2n,\bb C))$. Fix $m$, let $G_1=G_1(m)$, and
let  $\pi_1=L(\mu_1,\nu_1)$ be an irreducible representation of $G_1$.

Define the integer $k=k[\mu_1]$ by writing $\mu_1=(\seq a
k,\seqzero;\epsilon)$ with $a_1\ge a_2\ge\dots\ge a_k>0$. 
Write $\nu_1=(\seq b m)$, and define the integer 
$0\le q=q[\mu_1,\nu_1]\le m-k$ to be the largest integer such that 
$2q-2+\tau,2q-4+\tau,\dots,\tau$ all occur (in any order) in 
$\{\pm b_{k+1},\pm b_{k+2},\dots,\pm b_m\}$. 
After possibly conjugating  by the stabilizer of $\mu_1$ in $W,$ we may
write  
$$
\begin{aligned}
\mu_1&=(
\overbrace{\seq a k}^k,\overbrace{\seqzero}^{m-q-k}
\overbrace{\seqzero}^q;\epsilon)\\
\nu_1&=( \overbrace{\seq b k}^k,
\overbrace{b_{k+1},\dots,b_{m-q}}^{m-q-k},
\overbrace{2q-2+\tau,2q-4+\tau,\dots,\tau}^q). 
\end{aligned}
$$
Let $\mu_1'=(\seq a k)$, $\nu_1'=(\seq b k)$, and 
$\nu_1''=(b_{k+1},\dots,b_{m-q})$.

Then for $n\ge n(\pi_1)=m-\epsilon q+\frac{1-\epsilon}2\tau$, 
$\Theta_n(\pi_1)=L(\mu_2,\nu_2)$, where
$$
\begin{aligned}
\mu_2&=
(\mu_1',\overbrace{1,\dots,1}^{\frac{1-\epsilon}2(2q+\tau)},
\seqzero)\\
\nu_2&=(\nu_1',
\overbrace{2q-1+\tau,2q-3+\tau,\dots,2\epsilon q+1+\epsilon\tau}
^{\frac{1-\epsilon}2(2q+\tau)},
\nu_1'',\\
&\qquad\qquad 2n-2m-\tau,2n-2m-2-\tau,\dots,
-\epsilon(2q+\tau)+2). 
\end{aligned}
$$
If $\la_L,\la_R$ are the parameter of $\pi_1,$ then the parameter of
$\Theta(\pi_1)$ is 
$$
\begin{pmatrix}
\la_L,q-\tau/2+1/2,\dots ,\ep q+\ep\tau/2+3/2,
n-m-\tau/2,\dots ,-\ep q-\ep\tau/2 +1\\
\la_R,q-\tau/2-1/2,\dots ,\ep q+\ep\tau/2+1/2,
n-m-\tau/2,\dots ,-\ep q-\ep\tau/2 +1
\end{pmatrix}
$$           
\end{theorem}

\noindent{\textbf Note}: The lowest K-type $\mu_1=(\seq a 
k,\seqzero;\epsilon)$ for $\pi_1$ is degree-lowest in $\pi_1$ if 
$\epsilon=+1.$ If $\epsilon=-1$ the degree-lowest K-type of 
$\pi_1$  is  
$$
(\seq a k,\obr{r}{{1,\dots,1}},\seqzero;\eta)
$$
where
$$
(r,\eta)= 
\begin{cases}
(2q+\tau,1) &\text{ if } 2q+\tau\le m-k,\\
(2(m-k)-2q,-1) &\text{ if } m-k<2q+\tau\le2(m-k)+\tau. 
\end{cases}
$$
\subsection{Main Result}\label{sec:maintheta} 
Restrict attention to the cases when the
nilpotent orbit $\CO$ has columns
\begin{description}
\item[(B)] $(m_0)(m_1,m_2)\dots (m_{2p-1},m_{2p})$ with
  $m_{2k}>m_{2k+1},$
\item[(C)] $(c_0,c_1)\dots (c_{2p-2},c_{2p-1})(c_{2p})$ with
  $c_{2j-1}>c_{2j},$
\item[(D)] $(m_0,m_{2p+1})(m_1,m_2)\dots (m_{2p-1},m_{2p})$ with
  $m_{2j}>m_{2j+1}$.
\end{description}
To each such nilpotent orbit we associate a sequence of dual pairs as follows.
Let $(V_k,\ep_k)$ be a symplectic space if $\ep_k=-1,$ orthogonal if
$\ep_k=1,$ $k=0,\dots ,2p.$  $\ep_0$ is the same as the type of the
Lie algebra, $\dim V_0$ is the sum of the columns. Let $(V_k,\ep_k)$
be the space with dimension  the sum of the lengths of the columns
labelled  $\ge k,$ and set $\ep_{k+1}=-\ep_k.$ Then 
\begin{equation*}
(V_k,V_{k+1})
\end{equation*}
gives rise to a dual pair. 
\begin{theorem}\label{t:corresp}
The unipotent
representations attached to $\CO_k$ are all $\Theta-$lifts of the
unipotent representations attached to $\CO_{k+1}.$
\medskip
More precisely, it is enough to describe the passage from $\CO_1$ to $\CO_0.$
\begin{itemize}
\item The infinitesimal character for $\CO_0$ is obtained from
  $\la_{\CO_1}$ by the procedure in proposition \ref{p:infl}; the
  resulting infinitesimal character is $\la_{\CO_0}.$
\item $\ep_0=-1$. There is a $1-1$ correspondence between unipotent
  representations of $Sp(V_{1})$ attached to $\CO_{1}$ and unipotent
  representations of $SO(V_0)$ attached to $\CO=\CO_0.$
\item $\ep_0=1$. There is a $1-1$ correspondence between unipotent
  representations of $O(V_{1})$ attached to $\CO_{1}$ and unipotent
  representations of $Sp(V_0)$ attached to $\CO=\CO_0.$
\end{itemize}
\end{theorem}

\begin{proof}\ 
The relation between $\CO_1$ and $\CO_0$ is that one adds a column
longer than the longest column of $\CO_1.$ This adds one to the
existing rows of $\CO_1$ and adds some rows of size 1. When passing
from $sp(*)$ to $so(*),$ the component group acquires another $\bZ_2.$ 
When passing from $so(*)$ to $sp(*),$ the component group does not change.

If $\CO_1$ is type C then $\CO_0$ is type B, and we add a column $m_0$
which must be of odd length. The infinitesimal character is augmented by
$(m_0/2,\dots ,1/2)$ conforming to \ref{p:infl}. There are two cases:
\begin{enumerate}
\item $c_{2p}=0$. In this case $c_0\to m_1,\dots ,c_{2p-1}\to
  m_{2p}$. So the pairing of the columns of $\CO_0$ matches
  $(m_0)(c_0,c_1)\dots   (c_{2p-2},c_{2p-1})$ and $\Theta$ gives a 1-1
  correspondence between parameters for $\CO_1$ and $\CO_0$. 
\item $c_{2p}\ne 0.$ In this case, $c_{2p}$ is even. Again $c_0\to
  m_1,\dots, c_{2p-1}\to m_{2p},$ but $c_{2p}\to m_{2p+1}$ and we have
  to add $m_{2p+2}=0.$ The pairing of columns for $\CO_0$ is
  $(m_0)(c_1,c_2)\dots (c_{2p-2},c_{2p-1})(c_{2p},0).$ Since
  $c_{2p}>0$ is even, the last pair does not contribute any unipotent
  representations. 
\end{enumerate}
In both cases 
$$
(\eta_1,\dots ,\eta_p)\longleftrightarrow (\eta_1,\dots ,\eta_p).
$$ 
If the pair is  from type $C$ to type $D$, a column $m_0$ is added,
and the infinitesimal character matches Proposition \ref{p:infl}.
$c_0,\dots ,c_{2p}$ are changed to $m_1,\dots ,m_{2p+1}$ 
The pairing of the columns of $\CO_0$ is  
$$
(m_0,c_{2p})(c_1,c_2)\dots (c_{2p-2},c_{2p-1}).
$$ 
A parameter corresponding to a $(\eta_1,\dots ,\eta_p)$ goes to the
corresponding one with $(\eta_1,\dots ,\eta_p)$ for type $D$. 

The correspondence for parameters of type $B,D$ with type $C$ when the
lowest K-type is with a $+$ is analogous to type $C$ to type $B,D$
above. The cases when the lowest $K-$type is with a $-$ are as
follows. In all cases the infinitesimal characters conform to
proposition \ref{p:infl}.

For type $B$ to type $C,$ an odd column $c_0$ larger than $m_0$ is added,
and $m_0\to c_1,\dots , m_{2p}\to c_{2p+1},$ and we must add a
$c_{2p+2}=0.$ The pairing of the columns is
$$
(c_0,m_0)(m_1,m_2)\dots (m_{2p-1},m_{2p})(0).
$$
Theorem \ref{t:abmain} implies that the $\Theta-$lift of the 
parameter for $\CO_1$ corresponding to $(\eta_1,\dots ,\eta_{p})$ goes
to the parameter $\eta_0=-1,\eta_1,\dots ,\eta_p)$ for $\CO_0.$ The
parameters with $\eta_0=1$ are $\Theta-$lifts of the paramters of
$\CO_1$ with $K-$types with a $+$.

For $\CO_1$ of type $D$ to $\CO_0$ of type $C,$ an even  column $c_0$
larger than $m_0$ is 
added, and $m_0\to c_1,\dots ,m_{2p+1}\to c_{2p+2}$. The columns of the
ensuing $\CO_0$ are paired
$$
(c_0,m_0)(m_1,m_2)\dots (m_{2p-1},m_{2p})(m_{2p+1})
$$
Theorem \ref{t:abmain} implies that the $\Theta-$lift of the 
parameter for $\CO_1$ corresponding to $(\eta_1,\dots ,\eta_{p})$ goes
to the parameter $(\eta_0=-1,\eta_1,\dots ,\eta_p)$ for $\CO_0.$ The
parameters with $\eta_0=1$ are $\Theta-$lifts of the paramters of
$\CO_1$ with $K-$types with a $+$.

\end{proof}
We abbreviate $Sp(2n), O(m)$ for $Sp(2n,\bb C), O(m,\bb C)$ and
similarly for the Lie algebras. 
\begin{example}
Consider the nilpotent orbit in $so(8)$ with columns 
$\CO\longleftrightarrow (4,3,1)$. The infinitesimal character is
$(1,0,3/2,1/2).$ Then $(V_0,1)$ is of dimension 8, and $(V_1,-1)$ is of
dimension 4. $\CO_1\longleftrightarrow (3,1)$ and the unipotent
representations are the two oscillator representations 
\begin{equation*}
  \begin{pmatrix}
    3/2&1/2\\3/2&1/2
  \end{pmatrix}
\qquad   \begin{pmatrix}
    3/2&1/2\\3/2&-1/2
  \end{pmatrix}
\end{equation*}
They correspond to the two unipotent representations of $SO(8)$ with
parameters
\begin{equation*}
  \begin{pmatrix}
    1&0&3/2&1/2\\ 1&0&3/2&1/2
  \end{pmatrix}
\qquad   \begin{pmatrix}
    1&0&3/2&1/2\\ 1&0&3/2&-1/2
  \end{pmatrix}  
\end{equation*}
\end{example}
\begin{example}
Let $\CO_1\longleftrightarrow (4,2,2)$ in $so(8)$. It matches
$\CO_0\longleftrightarrow (4,4,2,2)$ in $sp(12)$. The infinitesimal
characters are $(1,1,0,0)$ and $(2,1,1,1,0,0).$ The parameters for
$\CO_1$ are
$$
\begin{pmatrix}
  1&0&1&0\\1&0&1&0
\end{pmatrix}
\qquad
\begin{pmatrix}
  0&-1&1&0\\1&0&1&0
\end{pmatrix}
$$
The parameters for $\CO$ are
$$
\begin{aligned}
&\begin{pmatrix}
  2&1&0&-1&1&0\\ 2&1&0&-1&1&0
\end{pmatrix}
&&\begin{pmatrix}
  2&1&0&-1&0&-1\\2&1&0&-1&1&0
\end{pmatrix}\\
&\begin{pmatrix}
  1&0&-1&-2&1&0\\2&1&0&-1&1&0
\end{pmatrix}&&
\begin{pmatrix}
  1&0&-1&-2&0&-1\\2&1&0&-1&1&0
\end{pmatrix}
\end{aligned}
$$
The second column is obtained by applying the correspondence to the
parameters for $O(8)$ tensored with $sgn.$ 
\end{example}

\begin{example}
Let $\CO\longleftrightarrow (33)$ in $sp(6).$ Then
$\CO_1\longleftrightarrow (3)$ in $so(3).$ The infinitesimal
characters are $(3/2,1/2,1/2)$ and $(1/2)$ as given by the previous
algorithms.

\bigskip
The rows of $\CO$ are $(2,2,2)$ there is only one special unipotent
representation, its infinitesimal character is $(1,1,0)$. By contrast
infinitesimal character $(3/2,1/2,1/2)$ matches the
$\Theta-$correspondence and there are two parameters. 
\end{example}
\begin{example}
Let $\CO\longleftrightarrow (4,2,2)$ in $sp(8).$ It corresponds to\newline
$\CO\longleftrightarrow (2,2)$ in $so(4).$ There are two such
nilpotent orbits if we use $SO(4),$ one if we use $O(4).$ We will use
orbits of the orthogonal group. The infinitesimal character
corresponding to $(2,2)$ is $(1/2,1/2)$. The representations
corresponding to $(4,2,2)$ have infinitesimal character
$(1,0,1/2,1/2).$  The Langlands parameters are spherical
$$
\begin{pmatrix}
  1/2&1/2\\1/2&1/2
\end{pmatrix}
\longleftrightarrow 
\begin{pmatrix}
  1&0&1/2&1/2\\1&0&1/2&1/2
\end{pmatrix}
$$

\bigskip
{\rm
We can go
further and match $(2,2)$ in $so(8)$ with $(2)$ in $sp(2).$ If we
combine these steps we get infinitesimal characters $(1)\mapsto
(0,1)\mapsto (2,1,0,1).$ There is
nothing wrong with the correspondence of irreducible modules. But note
that the infinitesimal character $(2,1,1,0)$ has maximal
primitive ideal corresponding to the orbit $\CO\longleftrightarrow
(4,4)$, (rows $(2,2,2,2)$).

{This is one of the reasons for imposing the conditions on the
  nilpotent orbits, we want to be able to iterate and stay within the class of
  unipotent representations.} One obtains 
induced modules with interesting composition series. In this example,
let $P$ be the parabolic subgroup with Levi component $GL(2)\times
Sp(4)$ and $\chi$ be a character on $GL(2)$ so that the induced module
$\Ind_P^{Sp(8)}[\chi\otimes Triv]$ has infinitensimal character
$\la_\CO,\la_\CO)$ with $\la_\CO=(2,1,1,0).$ Then  
\begin{equation*}
  \begin{aligned}
&\Ind_{GL(2)xSp(4)}^{Sp(8)}[\chi\otimes Triv]=\\
&\begin{pmatrix}
  2&1&0&-1\\
  2&1&2&-1
\end{pmatrix}+
\begin{pmatrix}
  2&1&0&-1\\
  1&0&-1&-2
\end{pmatrix}+
\begin{pmatrix}
  1&0&2&1\\
  0&-1&2&1
\end{pmatrix}.    
  \end{aligned}
\end{equation*}
The first two parameters are unipotent, corresponding to
$\CO\longleftrightarrow (4,4).$ The last factor is bigger, the
annihilator corresponds to the nilpotent orbit $(4,2,2)$.  
All these composition factors have nice character formulas
analogous to those for the special unipotent representations even though their
annihilators are no longer maximal. Daniel Wong has made an extensive
study of these representations in his thesis. 

{This example is tied up with the fact that
  nilpotent orbits are not always normal}.
A nilpotent orbit is normal if and only if $R(\CO)=R(\ovl{\CO}).$ 
The orbit  $(4,2,2)$ is \textbf{not normal}. 
$R(\CO)$ is the full induced representation from a \newline 1-dimensional
representation of $\fm=gl(2)\times sp(4)$.
$R(\ovl{\CO})$ is the sum of the first and last representation, 
missing the middle one. 
{ 
These equalities are in the sense that the $K-$types of the
  representations match the $G-$types of the regular functions, using the
  identification $K_\bb C\cong G$. 
It is \textbf{not} the case that $R(\CO)$ and $R(\ovl\CO)$ are
  representations of $G$ as a real Lie group}.  
}    
\qed
\end{example}

\section{Regular Functions on Nilpotent Orbits 
and Unipotent  Representations} \label{sec:5}

\vh
\subsection{Background Results}\label{S:1} 
Most of the details in this section can be found in 
\cite{McG}, \cite{G}, 
and references therein.
 
The structure sheaf of a variety $Z$ will be denoted by $\CS_Z.$ We
will abbreviate $R(Z)$ for $\Gamma(Z,\CS_Z).$ 

\vh
Typically $\CO$ will denote the orbit of a nilpotent element $e$ in a
reductive Lie algebra $\fg.$  The orbit is isomorphic to $G/G(e).$
Its universal cover $\wti\CO$ is isomorphic to $G/G(e)_0.$ By one of
Chevalley's theorems there is a representation  $\tilde V$ and a
vector $\tilde e=(e,\tilde v)\in \fg\oplus \tilde V$ such that its
orbit under $G$ is the universal cover; in other words the stabilizer
of $\tilde v$ is $G(e)_0$.  Given any closed subgroup $G(e)_0\subset
H\subset G(e),$ there is a corresponding cover $\wti\CO_H$ which can be
realized as the orbit of $G$ of an element $e_H=(e,v_H)\in\fg\oplus V_H.$  

\vh
\subsection{}\label{S:2}
Let $\{e,h,f\}$ be a Lie triple associated to $e.$ Let $\fg_{\ge 2}$
be the sum of the eigenvectors of $\ad h$ with eigenvalue greater than
or equal to $2.$ Let $P(e)$ be the parabolic subgroup determined by
$h,$ \ie the parabolic subgroup corresponding to the roots with
eigenvalue greater than or equal to zero for $\ad h.$ It is well known
that the natural map
\begin{equation}\label{5.1}
m_e : G\times_{P(e)} \fg_{\ge 2} \longrightarrow \ovr \CO,\qquad
(g,X)\mapsto gXg^{-1}
\end{equation}
is birational and projective. The birationality follows from
\cite{BV}. The projective property is in \cite{McG}. 

\vh
\subsection{}\label{S:3}
The notions and results in the next sections are for $G$ the rational
points of a reductive group over an algebraically closed field of 
characteristic 0. 

\medskip
Let $P=MN$ be an arbitrary parabolic subgroup and $\CO_\fm\subset \fm$
be a nilpotent orbit. A $G$-orbit $\CO$ is called {\it induced }
from $\CO_\fm$ (\cite{LS}), if
\begin{equation}\label{5.2}
\CO\cap [\CO_\fm + \fn]\quad \text{ is dense in }\quad \CO_\fm +\fn.
\end{equation}
Let $\Sigma:=\CO_\fm +\fn.$ There is a similar {\it moment map}
\begin{equation}\label{5.3}
m:G\times_P \Sigma \longrightarrow \ovr\CO,\qquad (g,X)\mapsto gXg^{-1}.
\end{equation}
It is projective for the same reason as before, but it is not always
birational. Precisely, if $e\in\Sigma\cap\CO,$ then
the generic fiber of $m$ is isomorphic to $G(e)/P(e).$ 

We will write $\CZ$ for $G\times_P\Sigma$ where $\Sigma=\CO_{\fm}+\fn.$ 
In general, write $A_G(e):=G(e)/G(e)_0$. We suppress the subscript $G$ if it is
clear from the context. Recall from \cite{LS} that $G(e)_0=P(e)_0,$
so that there is an inclusion  $A_P(e)\subset A_G(e).$ If $e_\fm\in
\CO_\fm$, then there is a surjection $A_P(e)\to A_M(e_\fm)$. Given a
representation $\phi$ of $A_M(e_\fm),$ we will denote by the same letter
$\phi$ its inflation to $A_P(e).$ 

\vh
\subsection{}\label{S:4}
Given a (cover of) an orbit $\C O\cong G/G(e),$ recall from \cite{J}
section 8.1 that
\[
R({\C O})=\Ind_{G(e)}^G[\Triv]\qquad \text{ (algebraic induction). }
\]
\begin{definition}
Let $\Psi\in \wht{G(e)}$ be trivial on $G(e)_0$, and write 
$$
R(\C O)_\Psi=\Ind_{G(e)}^G[\Psi].
$$    
\end{definition}
Regular
functions on the universal cover $\wti{\CO}$ satisfy
\[
R(\wti{\C O})=\sum_{\Psi\in\wht{G(e)_0} }\Ind_{G(e)}^G\Ind_{G(e)_0}^{G(e)}[\Psi].
\]

\subsection{}\label{S:5}
Let $e_\fm\in\CO_\fm$ and
$\la\in\fm$ be semisimple such that $C_\fk g(\la)=\fm,$ and $\fk n$ is
spanned by the root vectors of roots positive on $\la.$ 
Let $e\in e_\fm+\fn$ be a representative for the induced nilpotent. Let
$\psi$ be a representation of $A_M(e_\fk m)$ (equivalent to the inflated
representation on $A_P(e)$) and $\Psi$ be the induced 
representation to $A_G(e)$. Choose a ($K-$invariant) inner product on $\fg.$
By Frobenius reciprocity, 
$$
R(\CO)_\Psi:=\sum_{\rho\in \wht{A(\CO)}} [\rho\mid_{A_M(e_\fk
  m)}:\psi]R(\CO)_\rho.
$$
\begin{proposition}\label{p:5.5} Let $(\mu,V)$ be a representation of $G.$ Then
$$
[\mu:Ind_M^G[R(\CO_\fm)_\psi]]\le [\mu: R(\CO)_\Psi].
$$
\end{proposition}

\begin{proof} We work with $G(e)$ and $M(e_\fk m)$ with $\psi$ and
  $\psi$ trivial on the connected component of the identity so in
  particular also trivial on the corresponding unipotent radicals.  
For $n\in\bN,$ consider $\frac{1}{n}\la +e_\fm.$   
There is $p_n\in K\cap P$ such that $\la_n:=Ad(p_n)(\frac{1}{n}\la
+e_\fm)=\frac{1}{n}\la +e.$ The centralizer of $\la-n$ has the same
dimension as the centralizer of $e.$ We show that for every $(\mu,V)$
linearly independent vectors transforming under $M(e_\fk m)$ according
to $\psi$ give rise to vectors transforming according to $\Psi$ under
$G(e).$ 

\medskip
For each $n,$ let $X_n^1,\dots, X_n^k$ be
an orthonormal basis of $C_{\fk g}(\la_n),$ the centralizer in $\fg$ of
$\la_n.$ We can extract a  subsequence such that the $X_n^i$ all
converge to an orthonormal basis of $C_{\fk g}(e).$ 
Now let $v_1^n,\dots , v_l^n$ be an orthonormal basis of the space of
fixed vectors of $C_{\fk g}(\la_n)$ in $V.$ We can again extract a subsequence
such that the $v_n^j$ all converge to an orthonormal set of vectors in
$V.$ Because $v_n^j$ are invariant under the action of the $X_n^i,$
their limits are invariant under an orthonormal basis of $C_{\fk g}(e).$ Using
Frobenius reciprocity, this proves the claim for the connected
components of the centralizers, \ie the corresponding statement for
$R(\wti\CO_\fm)$ and $R(\wti\CO).$  The proof of the general case is a
straightforward modification, using hte fact that $A_G(e)$ and
$A_M(e_\fk m)$ are finite groups.
\end{proof}

\subsection{}\label{S:6}
For the case of a Richardson nilpotent orbit, the previous result can
be sharpened as follows. The details are in \cite{J} chapter 8.
Let $P=MN$ be a parabolic subgroup with Lie algebra $\fk p=\fk m +\fk
n.$ Denote again by $\la\in\fg$  a semisimple  element whose centralizer is
$\fk m,$ and which is positive on the roots of $\fk n.$ Let $e\in \fk
n$ be a representative of the Richardson induced orbit from this
parabolic subalgebra, and denote its $G$ orbit by $\CO.$ As before,
there is a map
\begin{equation}
  \label{eq:1.1}
m:  \C P_\fk n:=G\times_P\fk n\longrightarrow \fg,
\end{equation}
with image $\ovl{\CO}.$ Let $\wti{\CO}$ be the inverse image of
$\CO$. By lemma 8.8 in \cite{J}, $\wti{\C O}$ is a single $G-$orbit,
and an open dense subset of $\C P_\fk n.$ In addition $\wti{\C O}$ is
an unramified cover of $\C O$  with fiber 
$A_G(\C O)/A_P(\C O).$ 
Identify representations of $A_P(e)$ and $A_G(e)$ with representations
of $G(e)$ by making them trivial on $G(e)_0.$ 
\begin{proposition}\label{p:induced}
$$
[\mu: Ind_P^G[triv] ] = \sum_{\rho\in\widehat{ A_G(e)}}
[\rho|_{A_P(e)}:triv][\mu:R(\CO)_\rho]. 
$$
\end{proposition}
\begin{proof} By formula (4) in section 8.9 of \cite{J},
\[
R[\wti{\C O}]\cong R[\C P_\fk n]\cong \bigoplus_n H^0[G/P,S^n(\fk n^*)],
\] 
where $\C P_\fk n:=G\times_P\fk n.$ Theorem 8.15 in \cite{J} says that
\[
H^i[G/P,S^n(\fk n^*)]=0\qquad \text{ for all } i>0,\  n\in\bN.
\] 
The final formula follows by the standard relations between $H^i(G/P,V)$
and $\fk n-$cohomology. 
\end{proof}
We will use this proposition in the setting of a triangular nilpotent
orbit in a classical type Lie algebra, and $P$ such that  $A_P(e)=\{1\}.$

\subsection{}\label{S:7}
We return to the case where $P$ corresponds to the middle element
of the Lie triple. Write $A(\C O):=A_G(e).$ In this case,
$G(e)\subset P.$
Recall $\Sigma:=\CO_\fk m +\fk n$ and $\C Z:=G\times_P\Sigma.$ 
Let $\chi\in\widehat{A(\CO)}$ be an irreducible representation viewed as a
representation of $G(e)$ trivial on $G(e)_0,$ and assume there is a
representation $\xi$ of $P$ such that $\xi |_{G(e)}=\chi.$ Then 
\begin{equation}\label{5.9}
H^0(G/P,R(P\cdot e)\otimes \bC_\xi)\subset R(\CO,\CS_\chi)
\end{equation}
because $\CO$ embeds in $\CZ$ via $g\cdot e \mapsto [g,e].$ The results
in \cite{McG} imply that there is equality. Indeed, if $\phi\in
R(\CO,\CS_\chi),$ view it as a map $\phi:G\longrightarrow \bC$
satisfying
\begin{equation}
  \label{eq:5.10}
  \phi(gx)=\chi(x^{-1})\phi(g).
\end{equation}
 Then define a section $s_\phi\in H^0(G/P,R(P\cdot e))$ by the formula
 \begin{equation}
   \label{eq:5.11}
s_{\phi,\xi}(g)(p\cdot e):=\xi(p)\phi(gp).
 \end{equation}
The inverse map is given by 
\begin{equation}
  \label{eq:5.12}
s\mapsto  \phi_s(g):=s(g)(e).
\end{equation}
There is another inclusion
\begin{equation}
  \label{eq:5.13}
  H^0(G/P,R(\fg_{\ge 2}\otimes \bC_\chi)\subset H^0(G/P,R(P\cdot
  e)\otimes \bC_\xi).
\end{equation}
In \cite{McG} it is shown that when $\chi=triv$ and $\xi=triv,$ then
equality holds in (\ref{eq:5.13}), and in addition
\begin{equation}
  \label{eq:5.14}
  H^i(G/P,R(g_{\ge 2}))=(0) \quad\text{ for } i>0.
\end{equation}

These results suggest the  following conjecture.
\begin{conjecture}\label{5.4} For each $\chi\in \widehat{A(\CO)}$ there
is a representation $\xi$ of $P(e)$ 
{satisfying $\xi\mid_{G(e)}=\chi$} such that
\begin{equation*}
H^i(G/P(e),R(\fg_{\ge 2})\otimes \CS_\xi)=
\begin{cases}
R(\CO)_\chi, &\text{ if } i=0,\\
0 &\text{ otherwise. }
\end{cases}
\end{equation*}
\end{conjecture}

\vh 

\subsection{}
Recall Lusztig's quotient of the component group $A(\C O)$ denoted
$\ovl{A(\C O)}$, and the definition of cuspidal and stably trivial orbits.  
\begin{definition}
An orbit is called {cuspidal} if it is
not induced from any nilpotent orbit in a proper Levi component. A
more common terminology is rigid.

\medskip
A special orbit satisfying $A(\CO)=\overline{A(\CO)}$ is called
smoothly cuspidal.
\end{definition}
Smoothly cuspidal orbits have the property that the dual orbit 
$\check{\C O}$ is even. They are listed below, and not necessarily cuspidal.

\medskip
Let $\check e,\check h,\check f$ be a Lie triple associated to $\check{\C O}.$ 
In these cases $\la_\CO=\check h/2$. So these are the  parameters
treated in \cite{BV2}, also referred to as \textit{special
unipotent}. 

Precisely, in terms of partitions, smoothly cuspidal orbits for
classical groups are as follows.
\begin{description}
\item[(B)] Every row size except the largest one occurs an even number
  of times. Also the columns are $(m_0)(m_1,m_2)\dots (m_{2p-1},m_{2p})$ with
  $m_{2k}>m_{2k+1},$ and all columns have odd size.
\item[(C)] Every row size occurs an even number of times. Also 
the columns are $(c_0,c_1)\dots (c_{2p-2},c_{2p-1})(c_{2p})$ with
  $c_{2j-1}>c_{2j},$ and all column sizes are even.
\item[(D)] Every  row size occurs an even number of times. Also the
  columns are $(m_0,m_{2p+1})(m_1,m_2)\dots (m_{2p-1},m_{2p})$ with
$m_{2j}>m_{2j+1}$, and all column sizes are even.
\end{description}
\subsection{} View the complex
group $G$ as a real Lie group, and let $K$ be the maximal compact
subgroup.  Then $R(\CO)$ can be thought of as a $K$-module using the
identification of $K_c$ with $G.$  

Given $\chi\in \widehat {A(\CO)},$ denote by $R(\CO)_\chi$ the regular
sections of the sheaf corresponding to $\chi.$  We summarize the
statements in the paper in the following conjecture, which can be
thought of as a sharpening of the material in Section
\ref{sec:upar}. Recall the notion of \textit{Associated Variety} from
\cite{V}. A review of these notions and relations to the
\textit{Associated Cycle} is in later sections. 

We identify nilpotent $G-$orbits in the real algebra $\fk g$ with
$K_c-$orbits in $\fk p_c$ via the Kostant-Sekiguchi correspondence. 
Using the identification $\fk g_c=\fk g\times \fk g$ and $K_c\cong
G,$ $K_c-$orbits in $\fk p_c$ are identified with $G-$orbits in $\fk
g$, this time considered as a complex group and complex vector space
respectively.

\begin{conjecture}\label{conj:la} Given a nilpotent orbit $\CO,$ there is
an infinitesimal character $\la_\CO$ with the following property.

There is a 1-1 correspondence $\chi\longleftrightarrow X_\chi$ 
between characters of the component group and irreducible $(\fg,K)$
modules with ${\CO}$ as associated cycle and infinitesimal character $\la_\CO$
with the following properties:
\begin{enumerate}
\item The analogous character formulas as in \cite{BV2} hold,
\item $X_\chi$ are unitary,
\item $X_\chi |_K\cong R( \CO)_\chi$
\end{enumerate}
\end{conjecture}
The $\la_\C O$ given in section \ref{sec:upar} satisfy (2) by
the unitarity results in \cite{B1}; the character formulas in (1) 
are generalizations of those in \cite{BV2} using the Kazhdan-Lusztig
conjectures for nonintegral infinitesimal character (also in \cite{B1}).
\subsection{}\label{S:cg2}
 
Theorem \ref{t:5.10} below provides evidence  for (3).
\begin{theorem}\label{t:5.10}  Assume $\CO$ is smoothly
  cuspidal, and $\fk g$ is of classical type. 
There is a correspondence $\chi\longleftrightarrow X_\chi$ between
characters of ${A(\CO)}$ and   unipotent representations
determined by the property
  \begin{equation*}
    X_\chi\mid_{K}\cong R(\CO)_\chi.
  \end{equation*}
\end{theorem}
By the results in \cite{B1}, the representations $X_\chi$ are unitary as well.

The proof will be given in Section \ref{S:cg4}.
 
\subsection{}\label{S:8}
For the spherical case, theorem \ref{t:5.10} is more general.
\begin{theorem}
\label{t:sph}
Assume $\C O$ is arbitrary, $\fk g$ is of classical type, and let
$L_{triv}$ be the spherical module with infinitesimal character $\chi$
defined in sections  \ref{sec:upar} to \ref{sec:d}. Then
\[
R(\C O)\cong L_{triv}\mid_{K}
\]
\end{theorem}
These theorems imply that for the case of a
complex classical group, $R(\C O)\chi$ is realized as the $K-$spectrum of a
$(\fk g,K)-$module. In particular, $R(\C O)$ can be written as a
combination of standard modules with the same infinitesimal character,
not just as a combination of tempered modules as in \cite{McG}.

\subsection{Associated Cycle of an Admissible Module}
\label{S:10}  
We review the results in \cite{V1} and \cite{V2} which will be crucial
for the proof of the above theorems. Denote by $\C M(\fk g,K)$ the
category of admissible $(\fk g, K)-$modules.

\vh 
Recall $\fk g=\fk k+\fk s$ the
complexification of the Cartan decomposition of a real reductive
algebra $\fk g_0=\fk k_0+\fk s_0$. Let $(\pi,X)$ be an admissible $(\fk g, K)-$module. 
Section 2 of \cite {V1}, attaches to $(\pi,X)$ a  
$(S(\fk g),K)$ module $(gr(\pi), gr(X))$. This module is finitely
generated, graded.  
Attached to any $S(\fk g)-$module $M$ (equal to $gr(X))$ are varieties
\begin{equation}
  \label{eq:chvar}
\C V(M)\supset Supp (M)\supset Ass (M), 
\end{equation}
the set of prime ideals containing the
annihilator of $gr(M),$ the support of $gr(M)$, and the set of
\textit{associated primes}, those primes in $\C V (M)$ which are
annihilators of elements in $gr(M).$ Since $\fk k$ acts by zero,
$gr(X)$  is in fact an
$S(\fk g/\fk k)\cong S(\fk s)-$module. So the sets in (\ref{eq:chvar})
are all $K_\bC-$invariant varieties in $\fk s.$  Since the module $X$ was assumed
admissible, $M=gr(X)$ is finitely generated, so $\C V(M)=Supp (M)$,
and $Ass(M)$ is finite containing the minimal primes of $\C V(M).$ 
In particular the varieties corresponding to $Ass(M)$ and  $\C V(M)$ coincide.  
The center of $U(\fk g)$ must act by generalized
eigenvalues on an admissible module, so $S(\fk g)^\fk g$ acts by
0 on $M$. Thus the sets in (\ref{eq:chvar}) are contained in $\C N_\theta:=\C
N\cap \fk s$. We will write $\C V(X)$, $Supp(X)$, and $Ass(X)$ for the
corresponding objects for $M=gr(X).$

\vh
Denote by $\C C(\fk g,K)$ (definition 6.8 in \cite{V2}) 
the category of finitely generated $S(\fk
g/\fk k)-$modules $N$ carrying locally finite representations of $K,$
subject to 
\begin{equation}
  \label{eq:def68}
  \begin{aligned}
    &k\cdot(p\cdot n)=(\Ad(k)p)\cdot (k\cdot n)\qquad k\in K_\bC,\
    p\in S(\fk g),\quad n\in N,\\
    &\C V(N)\subset \C N_\theta.
  \end{aligned}
\end{equation}
Proposition 2.2 in \cite{V1} states that the map $gr$ gives rise to a
well defined map $Kgr$ between the Grothendieck groups $K\C M(\fk
g,K)$ to $K\C C(\fk g,K)$. Furthermore $X$ and $M=gr(X)$
have the same $K-$structure. 
  Choose representatives $\la_1,\dots ,\la_r$ for the nilpotent
$K_\bC-$orbits, and let $H_i$ be the corresponding isotropy subgroups. 
The support of any nonzero module $N\in\C C(\fk g,K)$ can be written
uniquely as a union of closures $K_\bC\cdot\la_i$ where $\la_i$ is not
in the cloure of any other orbit in the support. Following (7.4)(b)
and (7.4)(c) of \cite{V2}, let
\begin{equation}
  \label{eq:Csupp}
  \begin{aligned}
 &\C C(\fk g,K)_i:=\{ N\in\C C(\fk g,K)\mid
 \la_i\in( \ovl{K_\bC\cdot\la_j)}\backslash K_\bC\cdot\la_j\Rightarrow
 \la_j\notin\C V(N)\},\\
 &\C C(\fk g,K)_i^0:=\{ N\in\C C(\fk g,K)\mid \la_i\notin\C V(N) \}.
\end{aligned}
 \end{equation}
\begin{theorem}
  [2.13 in \cite{V1}, proposition 7.6 in \cite{V2}]
\label{t:av}
Attached to any $N\in\C C(\fk g,K)$ there is a genuine virtual representation
$\chi(\la_i,N)$ of $H_i$ with the following property.

This correspondence descends to an isomorphism of Grothendieck groups
\[
K\C C(\fk g,K)_i/K\C C(\fk g,K)_i^0\cong K\C F(H_i)
\]
where $K\C F(H_i)$ is the Grothendieck group of (algebraic)
representations of $H_i.$  
\end{theorem}

\begin{proposition}
  [proposition 7.9 in \cite{V2}]
Suppose that $(\tau,V_\tau)$ is an irreducible repreentation of $H_i.$
There is an object $N(\la_i,\tau)\in \C C(\fk g,K)$ such that:
\begin{enumerate}
\item $\C V((N,\la_i,\tau))=\ovl{K_\bC\cdot\la_i}.$
\item $\chi(N(\la_i,\tau))=\tau.$
\end{enumerate}  
Any such choice of $\{N(\la_i,\tau)\}$ gives rise to  a basis
$[N(\la_i,\tau)]$ of $K\C C(\fk g,K).$ 

When $\ovl{K_\bC\cdot\la_i}$ has no orbits of codimension 1 in its
  closure, one can choose
\[
N(\la_i,\tau)=\Ind_{H_i}^K \tau.
\]  
\end{proposition}
\begin{corollary}[4.11 and 4.7 in \cite{V1}, and \cite{V2}]
Assume that $G$ is the real points of a complex reductive group. A
basis of $K\C C(\fk g, K)$ is formed of 
\[
\big\{\ \Ind_{H_i}^K \tau\ \big\}_{i=1,\dots r,\tau\in\wht H_i}.
\]
The support of any irreducible $(\fk g, K)-$module $X$ is the closure
of a single orbit $\C O.$ Furthermore,
\[
X\mid_K=\Ind_{H_i}^K \chi(gr(X),\C O) - D(X)
\]
where $D(X)\in\C C(\fk g,K)$ with support strictly smaller than $\C O.$ 
\end{corollary}

\begin{definition}
 The \textbf{associated cycle} $AC(X)$ of an admissible $(\fk
g,K)-$module $X$ is the formal sum 
$$
AC(X):=\sum (\dim\chi_i)\ \C O_i
$$ 
where $\C V(gr(X))=\cup \ovl{\C O_i},$ are the irreducible components,
and $\chi_i=\chi(gr(X),\C O_i).$  $\dim\chi_i$ is
called the multiplicity of $\C O_i$ in the associated cycle of $X$. 
\end{definition}

\subsection{Asymptotic Cycle for Induced Modules}
\label{S:11} 
We follow \cite{BV1} section 3. Let 
\[
\Omega:=\{ X\in\fk g : |Im\la|<\pi, \text{ for any eigenvalue } \la
\text{ of } \ad X \}.
\]
Then $\Omega$ is invariant under $\Ad G,$ and there is an open
neighborhood $\C V,$ of the identity $e\in G$  such that  
$\exp:\Omega {\longrightarrow} \C V$ is an
isomorphism.

Next define functions
\[
\begin{aligned}
&j(X)=\det\big[\frac{e^{\ad X/2}-e^{-\ad X/2}}{\ad X}\big],\\
&\xi(X):=j(X)^{1/2},\ \xi(0)=1.  
\end{aligned}
\]
The Haar measure $dx$ on $G$ is related to Lesbegue measure on $\fk g$
by $dx=\xi(X)^2\;dX.$  There is a map 
\begin{equation}
  \label{eq:fphi}
  \begin{aligned}
&\phi\in C_c^\infty (\Omega)\mapsto f_\phi\in C_c^\infty (G)\\
&f_\phi(\exp X):=\xi(X)^{-1}\phi(X).
  \end{aligned}
\end{equation}
This induces a map on the level of distributions
\begin{equation}
\label{eq:dfphi}
\begin{aligned}
&\Theta\in\C D(\C V)\mapsto \theta\in \C D(\Omega),\\    
&\theta(\phi):=\Theta(f_\phi).
\end{aligned}
\end{equation}
This map takes $G-$invariant eigendistributions of the center of the
enveloping algebra $U(\fk g)$ to invariant eigendistributions on $\Omega$ of the
constant coefficient $G-$invariant operators  $\partial(I(\fk g))$ on
$\fk g_\bb C.$

\vh
Let $P=MAN$ be a parabolic subgroup, and $(\pi,\C H)$ an admissible
$(\fk m, M\cap K)-$module. 
For $\nu\in\fk a_\bC^*,$ where $\fk a:=Lie(A),$ let $\pi_P$ be the module equal to 
$\pi_P(man):=e^{(\nu-\rho)(\log a)}\pi(m),$ where
$\rho:=\frac12\sum_{\al\in\Delta(\fk n,\fk a)}\al.$  Let $\pi_\nu$ be the
induced module.
\begin{lemma}[Lemma 3.3 in \cite{BV1}]
  Let $\Theta:=\tr(\pi)$ and $\Theta_\nu:=\tr ({\pi_\nu}).$ Let $\phi\in
  C_c^\infty(P)$ and $f\in C_c^\infty(G).$ Then
\[
\begin{aligned}
&\tr\pi_P(\phi)&&=\int_{MAN} e^{(\rho+\nu)(\log a)}\Theta (m)\phi(man)\;dm\;da\;dn,\\
&\Theta_\nu(f)&&=\int_{MA}  
e^{(\rho+\nu)(\log a)}\Theta(m)\int_{KN}f(kmank^{-1})\;dk\;dn\;dm\;da=\\
&&&=\int_{MA}e^{\nu(\log a)}\Theta(m) D(ma)\int_{G/MA}f(xmax^{-1})\;dx\;dm\;da,
\end{aligned}
\]
where  $ma=exp(X_\fk m +X_\fk a)$ and  
$$
D(\exp(X_\fk m +X_\fk a))=\big| \det\big(e^{\ad (X_\fk m+X_\fk
    a)}-e^{-\ad (X_\fk m+X_\fk a)}\big)\mid_{\fk n}\big|.
$$
\end{lemma}
Let $\theta$ and $\theta_\nu$ be the lifts of $\Theta$ and
$\Theta_\nu.$ Plug in $f=f_\phi:$
\begin{equation}
  \label{eq:indlift}
\begin{aligned}
\theta_\nu(\phi)=&\int_{\fk m + \fk a} e^{(\rho+\nu)(X_\fk a )}\Theta(\exp X_\fk m)
\xi(X_\fk m + X_\fk a)^{-1}\cdot\\
&\cdot\int_{K\times\fk n}\phi(\Ad k(X_\fk m+ X_\fk a+X_\fk
n))\;dk\;dX_\fk n\;dX_\fk m\;dX_\fk a.  
\end{aligned}
\end{equation}
Decompose $\xi=\xi_{MA}\cdot \xi_N,$ and  denote by
\begin{equation}
  \label{eq:fP}
\phi_P(X_\fk m +X_\fk a)=
\int_{\fk n}\int_K \phi(\Ad k(X_\fk m +X_\fk a +X_\fk n))\;dX_\fk n\;dk.
\end{equation}
Formula (\ref{eq:indlift}) becomes
\begin{equation}
  \label{eq:ind2}
\theta_\nu(\phi)=\int_{\fk m +\fk a} e^{(\rho+\nu)(X_\fk a)}\theta(\exp X_\fk m)
\xi_N(X_\fk m + X_\fk a)^{-1}\cdot\phi_P(X_\fk m +X_\fk a)\;dX_\fk m\;dX_\fk a.  
\end{equation}
Recall from \cite{BV1},  $\phi_t(X):=t^{\dim
  g}\phi(t^{-1}X).$ Then
\[
(\phi_t)_P=t^{-\dim\fk n}(\phi_P)_t.
\]
It follows that if the asymptotic expansion of $\Theta$ has leading term
$D_r,$ then the leading term of the
asymptotic expansion of $\Theta_\nu$ is $D_{r}(\phi_P)$, but at
degree $r+\dim\fk n.$ 

\vh 
Write $\fk g=\ovl{\fk n} +(\fk m +\fk a) +\fk n.$ Denote by 
$\C F_{\fk  g}$ and $\C F_{\fk m+\fk a}$ the Fourier transforms with respect to
the Cartan-Killing form of $\fk g$ and the Cartan-Killing form of $\fk
g$ restricted to $\fk m +\fk a$ respectively. 
Formula (\ref{eq:fP}) defines a map $\phi\in
C_c^\infty(\fk g)\longrightarrow \phi_P\in C_c^\infty (\fk m +\fk a)$.

\begin{lemma}
  \[
\C F_\fk g(\phi)_{\ovl{P}}=\C F_{\fk m+\fk a}\big(\phi_{{P}}\big)
\]
\end{lemma}

\vh
Recall from \cite{BV1} that the leading term of a character
$AS(\Theta)$ is a combination of Fourier transforms of Liouville
measures of nilpotent orbits, $AS(\Theta)=\sum c_j\wht{\mu(\C O_j)}.$  
We call $AS(\Theta)$ the \textbf{asymptotic cycle} of $\Theta.$ 
\begin{definition}
Let $D$ be a tempered $MA-$invariant homogeneous distribution. Denote
by $\Ind_P^G D$ the distribution 
\[
\Ind_P^G[D](\phi):=D(\phi_P).
\]  
\end{definition}
When $D$ is the invariant measure of a nilpotent orbit $\C O_\fk
m\subset\fk m,$ $\Ind_P^G D$ is a combination of invariant measures
supported on nilpotent orbits of $\fk g.$

\begin{corollary} Using the notation $\theta$ for the character of
  $(\pi_M,\C H)$ and $\theta_\nu$ for the induced character, 
suppose $AS(\theta)=\sum c_j \C F_{\fk m+\fk a}(\mu(\CO_{j,\fk m}))$. Then
$AS(\theta_\nu )= \sum c_j \C F_\fk g\big(\Ind_P^G[\mu(\C O_{j,\fk m})]\big).$
\end{corollary}
\subsection{Relation between AC and AS}
\label{S:12}
According to results of Schmid-Vilonen \cite{SV2}, the nilpotent $G-$orbits and
$K_\bb C-$orbits in the
formulas for $AC$ and $AS$ correspond via the Kostant-Sekiguchi
correspondence, and $c_i=\dim\chi_i.$ 

\subsection{}\label{S:13}
The comparison between the Liouville measures of induced nilpotent
orbits and the inducing data is done in \cite{B3}. The
analysis of the distributions $\theta(f_P)$ is done in \cite{B4}
formula (8.3). Let $AS(\pi_M)=\sum_j c_j\C O_{j,\fk m}$. 
For each orbit $\C O_{j,\fk m}$ write $v_{ij}+X_{ij}$
for representatives of the orbits intersecting $\C O_{j,\fk m}+\fk n$
in open sets. Let $C_G(v_{ij})$ and $C_P(v_{ij})$ be the
centralizers. Then
\begin{equation}
  \label{eq:acformula}
  AS(\Ind_P^G \pi_M)=\sum_{i,j}
  c_j\bigg|\frac{C_G(v_{ij})}{C_P(v_{ij})}\bigg|\C O_{ij}.
\end{equation}

\section{Complex Groups}\label{S:cx}

\subsection{}\label{S:cg1}
We  specialize the results in section \ref{S:10} to \ref{S:13} 
 to the complex case. The main simplifications are that
 $AC(\pi)=c_\pi\C O_\fk m,$ and there is only one $\C O$ which
 intersects $\C O_{\fk m}+\fk n$ in a dense open set. We use $AC$ for
 both the asymptotic cycle and support identified via the
 Kostant-Sekiguchi correspondence. Formula (\ref{eq:acformula}) becomes
 \begin{equation}
   \label{eq:accx}
   AC(\Ind_P^G [\pi])=c_\pi\bigg|\frac{C_G(v)}{C_P(v)}\bigg|\C O.
 \end{equation}

\subsection{Proof of theorem \ref{t:5.10}}\label{S:cg4}
Since these results are clear for type A, we deal with types B,
C, D only.  We use the notation and parametrization in section \ref{sec:upar}.
Character identities are consequences of theorem III in \cite{BV2} and
its applications as detailed in \cite{B1}.

Assume first that $\C O$ is triangular corresponding to
$\{e,h,f\}$. Let $\{e^\vee,h^\vee,f^\vee\}$ be the dual nilpotent orbit in
$\vee \fk g.$ Let $M(h)$ be the centralizer of $h,$ $M(h^\vee)\subset G$
the centralizer of $h^\vee.$  By section 9 in \cite{BV2} on triangular
nilpotent orbits,  
every unipotent representation is induced irreducible from a character
of $\chi$ of $M(h).$ Parametrize the representations by these characters,
$\chi\in\widehat{M(h)}\longleftrightarrow X_\chi.$ From \cite{BV2}, the
passage from this parametrization to the one given by characters of the
component group of the dual nilpotent orbit is known explicitly. By
\cite{V1},
$$
X_\chi\mid_K=R(\C O)_{\rho(\chi)}-Y_\chi
$$
where $Y_\chi$ is a genuine $K-$module.  $\rho(\chi)$ is a
representation of the component group of the centralizer of $e,$
trivial on the unipotent radical because it is algebraic. Since
$\dim\rho(\chi)$ is also the multiplicity, it follows that
$\rho(\chi)$ must be $1-$dimensional. Since the reductive part of the
centralizer of $e$ is a product of classical groups, $\rho(\chi)$ is
trivial on the connected component. Thus $\rho(\chi)$ is a character
of $A(\C O).$ 

On the other hand, again by \cite{BV2},
$$
\Ind_{M(h^\vee)}^G[Triv]=\sum X_\chi.
$$ 
Using Proposition \ref{p:induced}, we get an identity 
$$
\sum R(\C O)_\psi=\sum X_\chi\mid_K=\sum R(\C O)_{\rho(\chi)}-\sum Y_{\chi}.
$$
It follows that 
$$
X_\chi\mid_K=R(\C O)_{\rho(\chi)}.
$$
It is clear that if $X_\chi$ is the spherical unipotent
representation, then $\rho(\chi)=Triv.$ 

Now let $\C O$ be a {\it special stably trivial} nilpotent orbit,
$\CO\subset\fg(n).$ The results in \cite{BV2} imply that there is a 1-1 correspondence
between characters of $A(\C O)$ and unipotent representations. 
Choose an arbitrary
parametrization of the unipotent representations  by characters of
$A(\C O)$, the trivial character should correspond to the spherical
module. As before, for each unipotent representation $X_\nu,$ there is a
representation $\rho(\nu)$ of the full centralizer of $e\in\C O,$ such that  
\begin{equation}
  \label{eq:6.3.1}
X_\nu=R(\CO)_{\rho(\nu)}- Y_\nu,  
\end{equation}
with  $Y_\nu$  a genuine $K$-module. 

Let $\fm=\fg(n)\times gl(k_1)\times\dots\times gl(k_r)$
be a Levi component of a parabolic subalgebra in
$\fg^+:=\fg(n + k_1+\dots +k_r).$ There are $k_1,\dots ,k_r$ such
that the orbit 
\begin{equation}\label{6.3}
\CO^+ = Ind_{\fm}^{\fg^+} [\CO\times triv\times \dots \times triv]
\end{equation} 
is {\it triangular.} 
Inducing $X_\nu$ up to $\fk g^+,$ and using the decomposition formulas
for such modules from \cite{BV2} combined with Propositions
\ref{p:5.5} and \ref{p:induced}, we
conclude as before that $\Ind Y_\nu=0$ so $Y_\nu=0,$ and the 
multiplicity of $X_\nu$ is 1. 
Thus $\rho(\nu)$ is a character of the component group $A(\C O),$ 
and counting occurences in the induced modules, we conclude that the
correspondence $\nu\longleftrightarrow\rho(\nu)$ is 1-1.    
In other words, there is a
parametrization $\nu\longleftrightarrow X_\nu$ such that 
\[
X_\nu=R(\CO)_\nu.
\]
\subsection{The correspondence $\psi\longleftrightarrow X_\psi$}
\label{sec:correspondence}
We  give details for  type C; the other types are similar. From section \ref{sec:upar} we know that the unipotent
representations are indexed by $(\ep_0,\dots,\ep_k),$ with $\ep_j=\pm
$ one for each pair of columns $(c_{2j},c_{2j+1})$. The component group
also has $k+1$ components,  $A(\CO)\cong\bZ_2^{k+1}$, 
one for each even size of rows. The sizes of
even rows are $(r_0,\dots , r_k).$ A character of $A(\CO)$ is given by an
$(\eta_0,\dots ,\eta_k)$, with $\eta_j=\pm$ according to whether the
character is trivial or not on the corresponding $\bZ_2^j.$ It is
enough to give the correspondence for the cases when all $\eta_i=+$
except for  one $\eta_j=-.$ The matching is that one sets
all the $\ep_s=-$ for the pairs of columns with label larger than or
equal to $j.$
The following Corollary
is key.
\begin{corollary}
  \label{c:orbit}
Let $\fk m\subset \fk g$ be a Levi component, $\CO_\fk m\subset \fk m$
a stably trivial special orbit, and $\CO=\Ind_\fk m^\fk g\CO_\fk m$
also stably trivial special. Then
$$
\Ind_\fk m^\fk g X_{\fk m,\nu}=\sum [\psi\ ;\ \fk \nu]X_{\fk g,\psi}
$$
\end{corollary}

\begin{example}
Consider the nilpotent orbit $\CO=(4422).$ The unipotent
representations are (writing $
\begin{pmatrix}
  \la_L\\\la_R
\end{pmatrix}
$ for the parameter) 
\begin{equation}
  \label{eq:ex24}
  \begin{aligned}
  &\pi(4_+,2_+) &&
  \begin{pmatrix}
    2&1&0&1&;&1&0\\
    2&1&0&1&;&1&0
  \end{pmatrix}\\
  &\pi(4_-,2_+) &&
  \begin{pmatrix}
    2&1&0&1&;&1&0\\
    1&0&-1&-2&;&1&0
  \end{pmatrix}\\
    &\pi(4_+,2_-) &&
  \begin{pmatrix}
    2&1&0&1&;&1&0\\
    2&1&0&1&;&0&-1
  \end{pmatrix}\\
  &\pi(4_-,2_-) &&
  \begin{pmatrix}
    2&1&0 &-1&;&1&0\\
    1&0&-1&-2&;&0&-1
  \end{pmatrix}  
  \end{aligned}
\end{equation}
The labeling $4_\pm,(2_\pm)$ indicates the $\ep$ on the columns of size
$2$ and $4$ respectively.

\vh
Write the rings of regular functions as $R(4^+2^+), R(4^-2^+),
R(4^+2^-), R(4^-2^-).$ Here the $(4^\pm2_\pm)$ indicate the $\ep$ on
the rows of size $4$ and $2$ respectively. 
Note that $\CO=(4422)$ is induced from
$(2211)\times triv$ of $sp(4)\times gl(3)$ and also from $(3322)\times
triv$ of $sp(10)\times gl(1).$  The partitions denote rows.

The composition series are
\begin{equation}
  \label{eq:sp4gl3}
  \begin{aligned}
    &
    \begin{pmatrix}
      2&1&0\\2&1&0
    \end{pmatrix}\otimes triv&&\longrightarrow
    \pi(4_+2_+)+\pi(4_-2_-)\\
&\\
    &R(2^+1)&&\longrightarrow R(4^+2^+)+R(4^+2^-)\\
&\\
    &
    \begin{pmatrix}
      2&1&0\\2&0&-1
    \end{pmatrix}\otimes triv&&\longrightarrow
    \pi(4_+2_-)+\pi(4_-2_+)\\
&\\
    &R(2^-1)&&\longrightarrow R(4^-2^+)+R(4^-2^-)
  \end{aligned}.
\end{equation}
and 
\begin{equation}
  \label{eq:sp10gl1}
  \begin{aligned}
    &
    \begin{pmatrix}
      2&1&0&-1&;&1\\2&1&0&-1&;&1
    \end{pmatrix}\otimes triv&&\longrightarrow
    \pi(4_+2_+)+\pi(4_+2_-)\\
&\\
    &R(32^+)&&\longrightarrow R(4^+2^+)+R(4^-2^+)\\
&\\
    &
    \begin{pmatrix}
      2&1&0&-1;&1\\1&0&-1&-2&;&1
    \end{pmatrix}\otimes triv&&\longrightarrow
    \pi(4_-2_+)+\pi(4_-2_-)\\
&\\
    &R(32^-)&&\longrightarrow R(4^+2^-)+R(4^-2^-)
  \end{aligned}.
\end{equation}
In these formulas, the nilpotent  $(2211)$ was abbreviated as $(21)$ with signs corresponding to the character on the rows of size 2, and $(3322)$ was abbreviated as $(32)$ with signs corresponding to the character on the rows of size 2.

\end{example}

\section{The Kraft-Procesi Model}

\subsection{Basic Setup}
We follow \cite{Bry}. Let $\CO$ be a nilpotent orbit given in terms of
the columns of its partition. Remove a column. The resulting partition
corresponds to a nilpotent orbit in a smaller classical Lie
algebra. The type alternates $C$ and $B/D.$ We get a sequence $(\fk
g_i,K_i)$ of classical algebras such that each $((\fk g_i,K_i),(\fk
g_{i+1},K_{i+1}))$ is a dual pair. Furthermore each pair is
equipped with an oscillator representation $\Om_i$ which gives the
Theta correspondence.  Form the $(\C G,K):=(\fk g_0,K_0)\times \dots
\times (\fk g_\ell,K_\ell))-$module 
$$
\Om:=\bigotimes \Om_i.
$$ 
The representation we are interested in, is the $(\fk g_0,K_0)-$module
$$
\Pi=\Om/(\fk g_1\times\dots\times \fk g_{\ell})(\Om).
$$
Let $(\fk g^1,K^1):=(\fk g_1\times\dots\times\fk
g_{\ell},K_1\times\dots\times K_{\ell}),$ and let $\fk g^1=\fk
k^1+\fk p^1$ be the Cartan decomposition. 

The following facts are standard. 
$\Pi$ is an admissible $(\fk g_0,K_0)-$ module. It has an infinitesimal
character compatible with the 
$\Theta-$correspondence, Proposition \ref{p:infl}.  
Furthermore the $K_i$ which are orthogonal
groups are disconnected, so the nontrivial component group  $\C K^1:=K^1/(K^1)^0$ 
 still acts, and commutes with the action of
$(\fk g_0,K_0)$. Thus $\Pi$ decomposes 
$$
\Pi=\bigoplus \Pi_\Psi
$$ 
where $\Pi_\Psi:=\Hom_{\C K^1}[\Pi,\Psi].$
The characters of $\C K^1$ are in $1-1$ correspondence with the
characters of $A(\C O)$ as in section \ref{sec:correspondence}.
\subsection{The Main Result} In the case of the
representations at the beginning of Section \ref{sec:maintheta}, the
ensuing representations are unipotent. 
Let $\C V:=\prod \Hom[V_i,V_{i+1}].$ This can be identified with a
  Lagrangian. Consider the variety 
$\C Z=\{(A_0,\dots  ,A_{\ell})\}\subset \C V$ given  by the equations 
$$
A_i^\star\circ
  A_i-A_{i+1}\circ A_i^\star=0, \dots ,A_{\ell+1}\circ
  A_{\ell}^*=0,\qquad i=0,\dots ,\ell-1.
$$  
\begin{theorem}[\cite{Bry}]
$\Om$ has a $(\C G,K)$ compatible filtration so that 
$$
gr(\Om/\fk p^1\Om)\cong R(\C Z). 
$$
$K^1$ still acts, and in particular 
$$
gr(\Om/\fk g^1\Om)_{\C K^1}\cong gr(\Om /\fk p^1\Om)_{K^1}\cong R(\ovl{\C O)}.
$$ 
\end{theorem}

Consider the coinvariants $R(\C Z)_{\fk k^1}.$ Then $\C K^1:=K^1/(K^1)^0$ 
acts, and so we conclude
$$
R(\C Z)_{\fk k^1}=\bigoplus_{\Psi\in\wht{\C K^1}}R(\C Z)_\Psi
$$  
\begin{corollary}
 Assume the nilpotent orbit $\C O$ satisfies the conditions at the
 beginning of Section \ref{sec:maintheta}. Then
$$
\Pi(\C O,\psi)\mid_{K^1}\cong R(\C Z)_{\psi}.
$$
\end{corollary}
\begin{remark}
$A(\C O)$ does not act on $\ovl{\C O}$, so we cannot identify\newline 
$[R(\C Z)_{\fk k^1}:\psi]$ with $R(\ovl{\C O})$ as in the case $\psi=Id.$ But
$K^1$ does act on $\C Z,$ so that the formula in the Corollary makes
sense. In the cases when $\ovl{\C O}$ is normal, it is reasonable to
conjecture that $R(\C O,\psi)\cong [R(\C Z)_{\fk k^1}:\psi].$ This
would follow from the conjecture that any regular function on the
inverse image of\newline 
$\C O\times\C O_1\times \dots \times\C O_\ell$ is regular on all of $\C Z.$ 

\end{remark}

\section{Beyond the Theta  Correspondence}

\subsection{} We consider the case of the $Spin$ groups of type $B_n\
D_n.$ We are concerned with \textit{genuine} unipotent
representations. In coordinates this means that the $K-$types have
half integer entries only. 

\begin{theorem}
A genuine representation $(\pi,V)$ is unitary  only if it is
induced from a representation $\pi_1\otimes\dots\otimes
\pi_k\otimes\pi_0$ on a Levi component $L=GL(m_1)\times \dots \times
GL(m_k)\times G_0$ where 
\begin{enumerate}
\item the representations $\pi_i$ for $i=1,\dots ,k$ are unitary with
  $1-$dimensional lowest $K-$types $(\mu_i+1/2,\dots ,\mu_i+1/2)$ with
  $\mu_i\in \bb N$, 
\item $\pi_0$ has lowest $K-$type $spin.$
\end{enumerate}
\end{theorem}
\begin{proof}
This is a standard \textit{bottom layer} argument. See \cite{Br} for
  this specific case, and \cite{B1} for the more general complex case.
\end{proof}
It is conjectured that the basic cases from which the unitary
dual is obtained via unitary induction and complementary series are
such that $\pi_0$ is unitary, and the infinitesimal character is
integral for a system of type $C_n\times C_n$ for type $B$ (coroots in
the Langlands dual),  and
$D_n\times D_n$  for type $D.$ We therefore concentrate on 
representations with lowest $K-$type $spin.$  The following is a
sharper conjecture about the basic cases, following the parametrization
in \ref{sec:maintheta}. We treat type $B$ in detail, case $D$ is analogous.

\subsection{Type B} The orbit $\CO$ has columns $(m_0',\dots ,m_{2p}')$ 
and let 
\begin{equation}\label{eq:gnilpotent}
(m_0)(m_1,m_2)\dots (m_{2p-1},m_{2p})\qquad m_{2i}=m_{2i+1}+1.
\end{equation}
The columns satisfying  $m'_{2j}=m'_{2j+1}$ were removed. The parameter
\begin{equation}\label{eq:genuine}
\begin{aligned}
m_{2j}\text{ odd }&\longleftrightarrow 
\begin{pmatrix}
\frac{m_{2j}-1}{2}&\dots &1&\frac{-m_{2j}+2}{2}&\dots &\frac{-1}{2}\\
\frac{m_{2j}-2}{2}&\dots &\frac{1}{2}&\frac{-m_{2j}+1}{2}&\dots &-1  
\end{pmatrix}\\
m_{2j}\text{ even }&\longleftrightarrow 
\begin{pmatrix}
\frac{m_{2j}}{2}&\dots &\frac{1}{2}&\frac{-m_{2j}+1}{2}&\dots &\frac{-1}{2}\\
\frac{m_{2j}-1}{2}&\dots &0&\frac{-m_{2j}+1}{2}&\dots &0  
\end{pmatrix}\\
m'_{2j}=m'_{2j+1}&\longleftrightarrow
\begin{pmatrix}
\frac{m_{2j}'}{2}&\dots &\frac{-m_{2j}'}{2}&\frac{-m_{2j}+1}{2}&\dots &\frac{m_{2j}'+1}{2}\\
\frac{m_{2j}'}{2}&\dots &\frac{-m_{2j}'-1}{2}&-m_{2j}'&\dots &\frac{m_{2j}'}{2} 
\end{pmatrix}
\end{aligned}
\end{equation}
is genuine. 
The infinitesimal character is $(\la_\CO,\la_\CO)$, same as in \ref{sec:maintheta},
but arranged so that 
$\disp{
\begin{pmatrix}
\la\\ w\la  
\end{pmatrix}
}$ 
has lowest $K-$type $spin.$  As before, the $m_{2j}'=m_{2j+1}'$
give rise to complementary series, and we concentrate on the case when there
are no such pairs. Note that  the orbit $\CO$
has an arbitrary number of rows of even size, while the odd sized rows
are $1,3,5,\dots ,4k+1$. 

The integral system for this parameter  (the coroots  with
integral inner produc with the prameter) form a system of type $C\times C.$ The corresponding
\textit{endoscopic group} is type $B\times B.$ 

\begin{proposition}\label{p:genuine}
There is a unique genuine parameter with infinitesimal character
$\la_\CO$ given by (\ref{eq:genuine}) and associated cycle a multiple
of  $\CO$ as in  (\ref{eq:gnilpotent}). 
(\ref{eq:genuine}). 
\end{proposition}
\begin{proof}
We use the generalized Kazhdan-Lusztig conjectures. It is enough to
consider one of the factors, $C_n$ in the integral roots of type
$C_n\times C_n$. The left and right maximal primitive ideals for 
part of $\la_L$ and $w\la_R$ correspond to what are called the Springer
and the Lusztig primitive ideal cell for the same nilpotent
orbit. These do not have any Weyl group representations in common
except for the special one, occurring with multiplicity 1. 
This is the uniqueness of the parameter. The rest of the argument is
as in \cite{B1}.  
\end{proof}

Denote by $A_{Spin}(\CO)$ the component group of the centralizer of an
$e\in\CO$ in the  $Spin-$group. Recall that $A(\CO)=\bb Z_2^{2k}.$
\begin{proposition}\label{p:cgroup} $A_{Spin}(\CO)$ is a nontrivial
  extension of $A(\CO)$ by $\bb Z_2$: 
$$
1\longrightarrow\bb Z_2\longrightarrow A_{Spin}(\CO)\longrightarrow
A(\CO)\longrightarrow 1. 
$$
In particular, $A_{Spin}(\CO)$ has $2^{2k}$ characters lifted from
$A(\CO),$ and one genuine character of degree $2^{k}.$ 

\end{proposition}
\begin{proof}
Let $(V,Q)$ be a quadratic space of dimension $2r+1$ with a basis
$\{e_i,v,f_i\}$ satisfying $Q(e_i,f_j)=\delta_{ij},$
$Q(e_i,v)=Q(f_j,v)=Q(e_i,e_j)=Q(f_i,f_j)=0$, and $Q(v,v)=-1.$ Let $C(V)$ be the Clifford
algebra with automorphisms $\al$ defined by 
$\al(x_1\dots x_r)=(-1)^r x_1\dots x_r$ and $\star$ given by $(x_1\dots
x_r)^\star=(-1)^r x_r\dots x_1.$ The double cover of $O(V)$ is 
$$
Pin(V):=\{ x\in C(V)\ \mid\ x\cdot x^\star=1,\ \al(x)Vx^\star\subset V\},
$$
and the double cover of $SO(V)$ by the elements in $Pin(V)$ which are
in $C(V)^{even}.$ 
The action of $Pin(V)$ on $V$ is given by $\rho(x)v=\al(x)vx^*.$ The
element $-I\in O(V)$ is covered by 
\begin{equation}
\label{eq:-spin}
\pm E_{2r+1}=\pm v\prod_{1\le i\le r} [(1-e_if_i)/\sqrt{-1}].  
\end{equation}
Suppose $V=V_{2i+1}\oplus V_{2j+1}$ is a quadratic space and direct sum of 
spaces of dimensions $2i+1, 2j+1$ so that  the restriction of the
quadratic form is nondegenerate on each of them. 
Then there are
two such operators, $E_{2i+1}$ and $E_{2j+1}.$ They satisfy the relations 
$$
\begin{aligned}
&E_{2i+1}E_{2j+1}=-E_{2j+1}E_{2i+1}\\
&E_{2r+1}^2=(\sqrt{-1})^{-r}.  
\end{aligned}
$$
Fix an element $\varepsilon\in\C O$. Its action on $V$ can be described in terms of
Jordan blocks. Because $\varepsilon$ is skew with respect to $Q,$ the
action on an odd sized block can be represented by a seqence of arrows 
$$
e_1\longrightarrow e_2\longrightarrow\dots\longrightarrow
e_r\longrightarrow v\longrightarrow f_r\longrightarrow
-f_{r-1}\longrightarrow\dots \longrightarrow (-1)^{r+1} f_1\longrightarrow 0,
$$
where the $e_i, f_j$ are in duality and $v$ has norm $1$. 
The group $A(\CO)$ is generated by even
products of elements each of which act by $-I$ on one of the odd
\textit{Jordan blocks}  of $\CO$, and $+I$ on the others. 
The  inverse image of $A(\CO)$ in $Spin(V)$ is generated by even
products of $\pm E_{2r+1}$ as in (\ref{eq:-spin}). 
\end{proof}

\begin{proposition}\label{t:spinunitary} 
The parameters in (\ref{eq:genuine}) are unitary.
\end{proposition}
\begin{proof}
 See \cite{Br}.
\end{proof}

Consider the special case of $\CO$ with columns $(2m+1,2m).$ The parameter is 
\begin{equation}\label{eq:genparso}
\begin{pmatrix}
\la_L\\ \la_R
\end{pmatrix}
=
\begin{pmatrix}
m,&\dots &1&-1/2,&\dots &-m+1/2\\
m-1/2&\dots &1/2&-1&\dots &-m
\end{pmatrix}
\end{equation}
The orbit $\CO$ has $SL_2-$triple $\{E,h,F\}$ with $h=(1,\dots ,1)$
and  $E$ with Jordan blocks 
$$
e_i\longrightarrow -f_i\longrightarrow 0.
$$
Let $\fk p:=\fk m+\fk n=C_h(0)+C_h(1)+C_h(2)$ be the parabolic subalgebra
corresponding to $h$, where $C_h(i)$ are
the $i-$eigenspaces of $h$. In particular $C_h(0)=\fk m=\cong gl(2m),$
and $\fk n=C_h(1)+C_h(2).$   
The centralizer of $E$ is $C_E=C_E(0)+C_h(1)+C_h(2),$ with
$C_E(0)\cong sp(2m,\bb C)\subset gl(2m)$ embedded in the standard
way. The component group of  the centralizer of $E$  in $SO(4m+1)$ is
trivial, while the centralizer in $Spin( 4m+1),\bb C)$ is $\bb Z_2.$
So there are two characters of $A_{Spin}(\C O)$, $\psi_{triv}$ and
$\psi_{gen}$. 

\begin{proposition} Let $V(\mu)$ denote a $K-$type with highest weight
  $\mu$. 
$$
\begin{aligned}
&R(\C O,\psi_{triv})=\sum V(a_1,a_1,\dots , a_m,a_m),\\
&
R(\C O,\psi_{gen})=\sum V(a_1+1/2,a_1+1/2,\dots , a_m+1/2,a_m+1/2).
\end{aligned}
$$
with $a_1\ge \dots \ge a_m\ge 0.$
\end{proposition}
\begin{proof}
Kostant's theorem implies that the $\fk n$ fixed vectors of
$V(\mu_1,\dots ,\mu_{2m})$ are the $gl(2m,\bb C)-$module  generated by
the highest weight. The vectors fixed by $sp(2m ,\bb C)$ follow by
Helgason's theorem. 
\end{proof}

\begin{corollary}
$$
\begin{aligned}
&X(\CO,triv)\mid_K\cong R(\CO,\psi_{triv}),
&X(\CO,gen)\mid_K\cong R(\CO,\psi_{gen})
\end{aligned}
$$
\end{corollary}
\begin{proof}
The first identity follows from the Theta Correspondence,
$X(\CO,triv)$ matches the trivial representation  on $Sp(2m,\bb C).$
It also follows from the arguments in \cite{McG1}. An extension of
this argument implies the second identity, noting that
$Spin\otimes Spin$ is a fine $K-$type for the appropriate  cover of
$So(2m+1,2m).$  In more detail, the  character formula for
$X(\CO,gen)$ is
$$
X(\CO,gen)=\sum_{w\in W(B_n\times B_n)}\ep(w)X(w\cdot(\la_L,\la_R),(\la_L,\la_R)). 
$$ 
Using induction in stages and restricting to $K,$ this matches the
formula for induction from $Spin\otimes Spin$ to $S[Pin(2m+1)\times Pin(2m)]$ to
$Spin(4m+1)$. Then pass to the real form, and note that
$Spin\otimes Spin$ is a fine $K-$type. 
\end{proof}

\begin{proposition}
The multiplicity of $\C O$ in $X(\C O,\psi_gen)$ in Equation
(\ref{eq:genuine}) is $2^p.$
Let $\psi_{gen}$ be the unique irreducible representation of dimension $2^p$ of $A_{Spin}(\C O).$
Then  
$$
X(\C O,\psi_gen)\mid_{K}=R(\CO,\psi_{gen})
$$ 
\end{proposition}

\begin{proof}
The proof is essentially the same as for the cases in Section
\ref{S:cg4}. The triangular orbits are replaced by the orbits with
rows $(1,3,\dots ,4k+1).$ The 
induced modules from the two parabolic subalgebras are both
irreducible. The induced from the parabolic subalgebra with Levi
component products of $GL$ gives multiplicity $2^p.$ For the induced
from the other parabolic subalgebra, the trivial representation is
replaced by the representation with orbit $\C O $ corresponding to the
columns $(2m+1,2m).$ Since $X(\C O,gen)$ is genuine, and $2^p$ is the
smallest possible for a genuine representation of $C_{Spin}(\C O)$
(this representation is trivial on the connected component),  
the proof from Section \ref{S:cg4} carries over. We omit further details. 
\end{proof}

\end{document}